 \documentclass[12 pt]{article}
\usepackage{amssymb, amsmath, amsfonts,amsthm}
\usepackage[a4paper,vmargin={3.5cm,3.5cm},hmargin={2.5cm,2.5cm}]{geometry}
\usepackage{mathrsfs} 
\usepackage[usenames,dvipsnames,svgnames,table]{xcolor}
\usepackage{url} 
\usepackage{enumitem} 
\usepackage{lineno} 
\usepackage{ mathtools} 
\mathtoolsset{showonlyrefs}

\usepackage[colorlinks,
            linkcolor=Blue,
            anchorcolor=red,
            citecolor=RoyalPurple,
            ]{hyperref} 

\usepackage{titlesec} 
\titleformat{\section}[block]
{\normalfont \large\bfseries}
{\thesection}{1.2em}{\bfseries}
\titleformat{\subsection}[block]
{\normalfont \normalsize \bfseries}
{\thesubsection}{0.9em}{\bfseries}
\titlespacing{\paragraph}{%
  0pt}{
  0.5\baselineskip}{
  1em}

\usepackage{comment}
\specialcomment{mynote}
{\begingroup --------------------------------------------------- \textcolor{red} }
{ --------------------------------------------------------------------------\endgroup }

\excludecomment{mynote}

\specialcomment{versionOU}
{\begingroup --------------------------------------------------- \textcolor{blue} }
{ --------------------------------------------------------------------------\endgroup }
\excludecomment{versionOU}

\usepackage{setspace}
\makeatletter
\def\th@plain{%
  \thm@notefont{}
  \itshape 
}
\def\th@definition{%
  \thm@notefont{}
  \normalfont 
}
\makeatother

\newtheorem{thm}{Theorem}[section]
\newtheorem{cor}[thm]{Corollary}
\newtheorem{lem}[thm]{Lemma}
\newtheorem{prop}[thm]{Proposition}
\newtheorem{rem}[thm]{Remark}
\newtheorem{eg}[thm]{Example}

\newtheorem{defn}[thm]{Definition}
\numberwithin{equation}{section}

\newcommand{\e}{{\rm e}} 
\newcommand{\N}{\mathbb{N}}

\newcommand{\R}{\mathbb{R}}
\newcommand{\U}{\mathbb{U}} 
\newcommand{\Utwo}{\mathbb{T}} 
 \newcommand{\vl}{ { \ell}} 
 \newcommand{\vr}{{ r}} 
\newcommand{\Rm}{{\vl}^{\infty}} 

\newcommand\multiset[1]{ \{\!\!\{  #1  \}\!\!\} } 

\newcommand{\Exp}[2][]{ {\rm E}_{#1}\left[ #2 \right]}
\newcommand{\Expcond}[3][]{{\rm E}_{#1}\left[ \left. #2 ~\right|~ #3 \right]}

\newcommand{\ind}[1]{\mathbf{1}_{\left\{ #1 \right\} } } 

\newcommand{\F}{\mathcal{F}}
\newcommand{\G}{\mathcal{G}} 

\newcommand{\eqdis}{\overset{d}{=}}

\newcommand{\Rdtwo}{[-\log 2, 0)} 
\newcommand{\rr}{\zz} 

\newcommand{\SM}{\mathcal M_{+}  } 
\newcommand{\MF}{\mathcal M_{f}  } 
\newcommand{\prm}[2]{\left\langle #1, #2 \right\rangle } 
\newcommand{\ms}{{\mathcal I}  } 

\newcommand{\mLevy}{\Lambda} 
\newcommand{\mbar}{ \bar{\Lambda}} 
\newcommand{\zz}{z} 
\newcommand{\Lp}[1][\xi]{ #1^{[p]}}

\newcommand{\Zs}{\mathbf{Z}} 
\newcommand{\mmu}{\Pi} 

\newcommand{\levb}{{\tt d}} 
\newcommand{\lev}{\bar{\levb}} 
\newcommand{\cutb}[2][\levb]{ {#2}^{\{#1\}}}

\newcommand{\Ls}{\mathbf{L}} 

\newcommand{\Deltaa}{D} 
\newcommand{\Deltab}{\lambda} 
\newcommand{\dtb}{\lambda} 

\newcommand{\ft}[2]{f \left( #1, #2 \right)  } 
\newcommand{\gt}{g} 

\newcommand{\cutc}[2][\epsilon]{ {#2}^{[#1]}}

\newcommand{\Pc}{ {\mathcal{P}} }
\newcommand{\Expc}[2][]{ {\mathcal E}_{#1}\left[ #2 \right]}
\newcommand{\Expcondc}[3][]{{\mathcal E}_{#1}\left[ \left. #2 ~\right|~ #3 \right]}
\newcommand{\Xs}{ {\mathbf{X}} } 
\newcommand{\Xc}{ {\mathcal{X}} } 
\newcommand{\Exps}[2][]{\mathbf{E}_{#1}\left[ #2 \right]}
\newcommand{\Ps}{ {\mathbf{P}} }
\newcommand{\Ys}{\mathbf{Y}} 
\newcommand{\Yc}{\mathcal{Y}} 

\newcommand{\A}{A} 
\newcommand{\xbar}{\bar{z}} 
 \newcommand{\muA}{P} 

\newcommand{\XOs}{ {\mathbf{X}^{(0)}} } 
\newcommand{\XO}{ {X^{(0)}}} 
\newcommand{\Fb}{\mathcal{G}} 
\newcommand{\Pb}{\mathbb{P}} 

\newcommand{\tildeX}{ Y} 
\newcommand{\tildeP}{ Q} 
\newcommand{\tildePc}{ \mathcal Q} 

\newcommand{\tb}{ \tau } 

\newcommand{\Ws}{\mathbf{W}} 

\begin{document}
\title{\textsc{Growth-fragmentation processes and bifurcators}}
\author{Quan Shi \thanks{Institut f\"ur Mathematik, Universit\"at Z\"urich,
 Winterthurerstrasse 190, CH-8057 Z\"urich, Switzerland \protect\\ email: quan.shi@math.uzh.ch} 
 \thanks{The author thanks Jean Bertoin for suggesting this research and for his guidance throughout the work. This work is supported by the Swiss National Science Foundation 200021\_144325/1.}
}
\date{}
\maketitle
 \begin{abstract} 
Markovian growth-fragmentation processes introduced by Bertoin model a system of growing and splitting cells in which the size of a typical cell evolves as a Markov process $X$ without positive jumps. We find that two growth-fragmentation processes associated respectively with two processes $X$ and $Y$ (with different laws) may have the same distribution, if $(X,Y)$ is a {\it bifurcator}, roughly speaking, which means that they coincide up to a bifurcation time and then evolve independently. Using this criterion, we deduce that the law of a self-similar growth-fragmentation is determined by a cumulant function $\kappa$ and its index of self-similarity.  
 \end{abstract}
{\bf 2010 Mathematics Subject Classiﬁcation:} 60G51, 60J80.


\section{Introduction}\label{sec:intro}
We consider the family of {\it Markovian growth-fragmentation processes} introduced by Bertoin \cite{Bertoin:growth}, see also \cite{CCF:GF, CampilloFritsch, Doumic, Fritsch} for related works. This stochastic model describes the evolution of a particle system, in which each particle may grow or decay gradually and split randomly into smaller pieces, independently of the other particles. 

It is convenient to describe it in terms of a cell population. The size of a typical cell evolves as a Markov process $X = (X(t), t\geq 0)$ with values in $[0,\infty)$, with c\`adl\`ag path and only negative jumps. The process $X$ also encodes the relationship between cell size and cell replication: at each jump time $t\geq 0$ of $X$ with $\Delta X(t) = X(t) - X(t-)<0$, a ``daughter'' cell with initial size $-\Delta X(t)$ is born, and the ``mother'' is still alive after this cell replication. Each daughter follows the same dynamics as the mother and evolves independently of the other cells.  Starting at time $0$ from a single cell with size $x>0$, we construct in this way a population of cells and thus define a process $\Xs=(\Xs(t), t\geq 0)$, where $\Xs(t)$ denotes the sizes of the cells alive at time $t\geq 0$. The process $\Xs$ is called a {\it (Markovian) growth-fragmentation process starting from $x$ associated with the cell process $X$}.

By construction, the law of $\Xs$ is determined by the law of $X$, however, growth-fragmentations driven by cell processes with different laws may have the same distribution. A first instance of such processes appears in Pitman and Winkel \cite{PitmanWinkel} with $X$ the exponential of the negative of a pure-jump subordinator (so-called {\it fragmenter} in \cite{PitmanWinkel}). The main purpose of this work is therefore to provide a sufficient condition for growth-fragmentations driven by different cell processes to have the same distribution. 
Our main result can be informally described as follows: 

{\it If there exists a coupling of (the distributions of) two cell processes $X$ and $Y$ which is a {\bf bifurcator}, in the sense that they almost surely coincide for a strictly positive time and evolve independently afterwards, then under some mild technical conditions, the growth-fragmentations driven respectively by $X$ and $Y$ have the same finite-dimensional distribution. }

This will be stated rigorously in Theorem \ref{thm:1}. The idea of bifurcator also goes back to \cite{PitmanWinkel}, which provides an explicit construction of bifurcators of fragmenters, as well as a characterization of the laws of all bifurcators of fragmenters. 

Therefore, to give a sufficient condition for two growth-fragmentations to have the same distribution, it suffices to understand when two cell processes can be coupled to form a bifurcator (in other words, when there exists a bifurcator whose two marginal distributions are the respective laws of these two cell processes). We do not have a complete answer to this question in general, however, we investigate a study of bifurcators for positive self-similar Markov processes, which further allows us to characterize the laws of growth-fragmentations driven by self-similar processes, so-called {\it self-similar growth-fragmentation processes}. 

Self-similar growth-fragmentations have been previously studied in \cite{Bertoin:growth} and have interesting applications: this model is connected with certain growth-fragmentation equations, see \cite{BertoinWatson}; besides, a distinguished case of self-similar growth-fragmentation appears as the re-scaled limit of the lengths of the cycles obtained by slicing random Boltzmann triangulations with a simple boundary at heights, see \cite{BCK:maps}.

In order to state our results, let us recall some basic facts about L\'evy processes, which are closely related to self-similar Markov processes; see e.g. \cite{Bertoin:Levy, Kyprianou}. 
Let $\xi$ be a L\'evy process with no positive jumps, which is often referred to as a {\it spectrally negative L\'evy process (SNLP)}. The SNLP $\xi$ is possibly killed at some independent exponential time. The distribution of $\xi$ is characterized by its Laplace exponent $\Phi: [0,\infty) \to \R$: 
$$\Exp{\e^{q \xi(t)}} = \e^{\Phi(q) t}, \quad \text{ for all } q, t\geq 0.$$
It is well-known that the convex function $\Phi$ is can be expressed by the L\'evy-Khintchine formula 
\begin{equation}\label{eq:LK}
\Phi(q) =-k + \frac{1}{2} \sigma^2 q^2 + c q + \int_{(-\infty,0)} \left( \e^{q \zz}-1 + q (1 - \e^{\zz})  \right) \mLevy  (d \zz), \quad q\geq 0, 
\end{equation}
 where $k\geq 0$ is the killing rate, $\sigma \geq 0$, $c \in \R$ and the L\'evy measure $\mLevy$ on $(-\infty, 0)$ satisfies 
\begin{equation}\label{eq:Levym}
 \int_{(-\infty,0)} (|\zz|^2 \wedge 1) \mLevy (d\zz) <\infty.
\end{equation}
Then we say $\xi$ is a SNLP with characteristics $(\sigma, c , \Lambda, k)$.
We also introduce $\kappa\colon [0,\infty) \to (-\infty, \infty]$ which plays an important role in this work: 
\begin{equation}\label{eq:kappa}
  \kappa (q) := \Phi(q) + \int_{(-\infty,0)} (1-\e^{\zz})^q \mLevy  (d\zz), \quad q\geq 0.
\end{equation}
So $\kappa\geq \Phi$. Note that $\kappa$ is convex and $\kappa(q)<\infty$ for all $q\geq 2$ because of \eqref{eq:Levym}. 
We stress that $\kappa$ does not characterize the law of $\xi$, see Lemma \ref{lem:triple}.  

Let $\XO := \exp(\xi)$, and we write by convention $\XO(t) = \partial$ if $\xi$ is killed before $t$, where $\partial$ denotes a cemetery point. Then the process $\XO$ is called a {\it homogeneous cell process}, which is a special case of self-similar process. Let $\tilde{X}^{(0)}:= \exp(\tilde{\xi})$, where $\tilde{\xi}$ is another SNLP with $\tilde{\kappa}$ defined as in \eqref{eq:kappa}, and write $\Xs^{(0)}$ and $\tilde{\Xs}^{(0)}$ for two growth-fragmentations associated with $\XO$ and $\tilde{X}^{(0)}$ respectively (with the same initial size of ancestor $x>0$), see Section \ref{sec:HGF} for their formal construction.
\begin{thm}[Homogeneous]\label{thm:2'}
The following statements are equivalent: 
  \begin{enumerate}[label=(\roman*)]
  \item $\kappa = \tilde{\kappa}$;
  \item $\XO$ and $\tilde{X}^{(0)}$ can be coupled to form a bifurcator;
  \item the homogeneous growth-fragmentations $\Xs^{(0)}$ and $\tilde{\Xs}^{(0)}$ have the same finite-dimensional distribution.
  \end{enumerate}
\end{thm} This result partially encompasses Proposition 5 and Corollary 25 in \cite{PitmanWinkel}. 
We hence say that the growth-fragmentation $\XOs$ is a {\it homogeneous growth-fragmentation process with characteristic $\kappa$}. The function $\kappa$ serves as {\it cumulant} for $\XOs$, in the sense that
$$\Exps{\sum_{x\in \XOs(t)} x^q} = \exp\left(\kappa(q) t\right) \quad \text{for all } q\geq 2 \text{ and } t\geq 0,$$ 
which is proved in Proposition \ref{prop:1}. 

In general, a self-similar cell process with index $\alpha \in \R$ is associated with a L\'evy process by Lamperti's representation \cite{Lamperti} as follows. Let us define a time-change by
\begin{equation}
  \tau_t^{(\alpha)} := \inf\left\{r\geq 0~:~ \int_0^r \exp(-\alpha \xi(s))ds \geq t \right\}, \quad t\geq 0,
\end{equation}
with the convention that $\exp(-\alpha \xi(s)) = 0$ if $\xi$ is killed before $s$. For every $x>0$, let us denote by $\muA_x$ the law of the process
\begin{equation}\label{eq:Lamperti}
 X^{(\alpha)}(t):= x \exp (\xi(\tau^{(\alpha)}_{tx^{\alpha}})), \quad t\geq 0,
\end{equation}
with the convention that $X^{(\alpha)}(t)= \partial$ for every $t \geq x^{-\alpha}\int_0^{\infty} \exp(-\alpha \xi(s))ds$. 
We know from  \cite{Lamperti} that for every $c>0$, 
\begin{equation}\label{eq:ss}
  \text{the law of } (c X^{(\alpha)}(c^{\alpha}t),t\geq 0) \text{ under } \muA_x \text{ is }\muA_{cx},
\end{equation}
so we call $X^{(\alpha)}$ a {\it self-similar cell process with index $\alpha$} \footnote{The way we define the index of self-similarity $\alpha$ is coherent with the theory of self-similar fragmentations. However, we stress that in the theory of self-similar processes, it is rather $-\alpha$ which is called the index of self-similarity.}. 
If $\alpha=0$, then we simply have 
$X^{(0)}= x \exp(\xi)$ under $P_x$, which is indeed a homogeneous cell process. 
 
For $\alpha\neq 0$, we further need to assume that 
\begin{equation}\label{eq:sspH}
  \text{there exists } q>0 \text{ with } \kappa(q)<0. 
\end{equation}  
Let us construct a growth-fragmentation $\Xs^{(\alpha)}$ associated with $X^{(\alpha)}$ starting from an ancestor cell with initial size $x>0$, then \eqref{eq:sspH} is a natural assumption that ensures the {\it non-explosion} of the growth-fragmentation $\Xs^{(\alpha)}$, which means that for every time $t\geq 0$ the elements of $\Xs^{(\alpha)}(t)$ are locally finite, see \cite{Bertoin:growth}. It is also known from a recent work \cite{BS} that if $\kappa(q)> 0$ for all $q\geq 0$ and $\alpha \neq 0$, then the growth-fragmentation $\Xs^{(\alpha)}$ explodes in finite time.     
Under \eqref{eq:sspH}, it is known from Theorem 2 in \cite{Bertoin:growth} that $\Xs^{(\alpha)}$ keeps the self-similarity: 
recall that $\Xs^{(\alpha)}$ starts from an ancestor with initial size $x$, then for every $c>0$, the law of $(c \Xs^{(\alpha)}(c^{\alpha}t),t\geq 0)$ is the same as a growth-fragmentation associated with $X^{(\alpha)}$ starting from $cx$. So we call $\Xs^{(\alpha)}$ {\it a self-similar growth-fragmentation with index $\alpha$}. 

Let us now present our main result for the self-similar case. 
Denote $\tilde{X}^{(\tilde{\alpha})}$ for the self-similar cell process of index $\tilde{\alpha}\in \R$ associated with $\tilde{\xi}$ by Lamperti's representation \eqref{eq:Lamperti} and let $\tilde{\Xs}^{(\tilde{\alpha})}$ be the growth-fragmentation driven by $\tilde{X}^{(\tilde{\alpha})}$. Suppose that the respective ancestors of $\tilde{\Xs}^{(\tilde{\alpha})}$ and $\Xs^{(\alpha)}$ have the same initial size $x>0$. 
\begin{thm}[Self-similar]\label{thm:2}
Suppose that \eqref{eq:sspH} holds for both $\kappa$ and $\tilde{\kappa}$, then the following statements are equivalent: 
  \begin{enumerate}[label=(\roman*)]
  \item $\kappa = \tilde{\kappa}$ and $\alpha=\tilde{\alpha}$;
  \item $X^{(\alpha)}$ and $\tilde{X}^{(\tilde{\alpha})}$ can be coupled to form a bifurcator; 
 \item the self-similar growth-fragmentations $\Xs^{(\alpha)}$ and $\tilde{\Xs}^{(\tilde{\alpha})}$ have the same finite-dimensional distribution.
  \end{enumerate}
\end{thm}  
Therefore, the law of the self-similar growth-fragmentation $\Xs^{(\alpha)}$ is characterized by $(\kappa, \alpha)$. Note that it follows immediately from the self-similarity that if $\tilde{\Xs}^{(\tilde{\alpha})}$ and $\Xs^{(\alpha)}$ have the same finite-dimensional distribution, then $\alpha=\tilde{\alpha}$.

Let us outline our proofs.
For the homogeneous case, Theorem \ref{thm:2'}, we provide a direct proof of the equivalence $(i) \Leftrightarrow (iii)$, by drawing a connection between homogeneous growth-fragmentations and {\it branching L\'evy processes} introduced in \cite{Bertoin:CF}. However, this proof cannot be easily extended to the self-similar case. 
Nevertheless, we can deduce the implication $(iii) \Rightarrow (i)$ in Theorem \ref{thm:2} from the self-similarity mentioned above and a study of martingales in self-similar growth-fragmentations in \cite{BCK:martingales}. 
Further, we can construct a bifurcator of $X^{(\alpha)}$ and $\tilde{X}^{(\tilde{\alpha})}$ when $\kappa = \tilde{\kappa}$ and $\alpha=\tilde{\alpha}$ by extending the approach of Pitman and Winkel \cite{PitmanWinkel} and using Lamperti's transformation, which means $(i) \Rightarrow (ii)$. This motivates us to establish the general sufficient condition, Theorem \ref{thm:1}, which is informally stated above. We hence get the implication $(ii) \Rightarrow (iii)$ and complete the proof. 

Besides the class of self-similar processes associated with L\'evy processes by Lamperti's transformations, the stationary processes driven by L\'evy processes, exponential Ornstein-Uhlenbeck type processes (see e.g. \cite{Sato}), are also natural examples for cell processes. The techniques developed in this paper also open the way to study the growth-fragmentations associated with exponential Ornstein-Uhlenbeck type processes, which will be discussed in a subsequent work.

\paragraph{Organization of the paper}
The rest of this work is organized as follows. We start with working on homogeneous growth-fragmentations in Section \ref{sec:CFP}. We first study the bifurcators of homogeneous processes, and then characterize the laws of homogeneous growth-fragmentations by using their connections with branching L\'evy processes. 

In Section \ref{sec:cell}, we first provide a non-explosion condition of general Markovian growth-fragmentations, then we introduce the notion of bifurcators for general cell processes and establish our main result, Theorem \ref{thm:1}, a general sufficient condition for two growth-fragmentations to have the same law. Applying Theorem \ref{thm:1}, we complete the proofs of Theorem \ref{thm:2'} and Theorem \ref{thm:2}. 

\section{The homogeneous case}\label{sec:CFP}
Throughout the rest of this work, we denote by $\xi$ and $\gamma$ two SNLPs with respective characteristics $(\sigma, c  , \Lambda , k )$ and $(\sigma_{\gamma}, c_{\gamma} , \Lambda_{\gamma}, k_{\gamma})$, and define $\kappa $ and $\kappa_{\gamma}$ respectively for $\xi$ and $\gamma$ as in \eqref{eq:kappa}.
We also define  
\begin{equation}\label{eq:xbar}
  \xbar:=   \log (1- \e^z), \quad  z\in (-\infty, 0), 
\end{equation}
so that $\e^z + \e^{\bar{z}}=1$. Note that $z \mapsto \xbar$ is an involution, i.e. $\bar{\xbar}=z$. For every L\'evy measure $\mLevy$, we write $\bar{\mLevy}$ for the push-forward measure of $\mLevy$ via the map $z \mapsto \xbar$. We remark that it follows from \eqref{eq:Levym} that 
\begin{equation}\label{eq:mbar}
 \mLevy((-\infty,-\log 2]) < \infty \quad \text{ and }\quad  \bar{\mLevy}([-\log 2, 0))< \infty.
\end{equation}

This section is concerned with growth-fragmentations driven by homogeneous cell processes, and our investigation is consist of two parts. We first depict the structure of the family of SNLPs that have the same $\kappa$ in Section \ref{sec:sw}, specifically, we show that they can be derived from each other by the {\it switching transformations}, which are introduced by Pitman and Winkel \cite{PitmanWinkel} to study the bifurcators of fragmenters. We next show that the law of a homogeneous growth-fragmentation associated with $\exp(\xi)$ is characterized by $\kappa$. In this direction, we recall the construction of branching L\'evy processes introduced by Bertoin \cite{Bertoin:CF} in Section \ref{sec:BLP} and then  build a connection between homogeneous growth-fragmentations and branching L\'evy processes in Section \ref{sec:HGF}.
These two results motivate us to extend the conception of bifurcator to general Markov processes and to study the relations between bifurcators and Markovian growth-fragmentations, which will become the object of investigation in Section \ref{sec:cell}.  

We will often appeal to the following relation between the SNLPs that have the same $\kappa$ in terms of their characteristics. 
\begin{lem}\label{lem:triple}
There is $\kappa  =  \kappa_{\gamma}$, if and only if 
\begin{equation}\label{eq:triple}
 \mLevy  + \mbar  = \mLevy_{\gamma} + \mbar_{\gamma}, \quad \sigma = \sigma_{\gamma}, \quad c + \int_{(-\infty,-\log 2)}(1- 2\e^{\zz}) \mLevy (d\zz)  = c_{\gamma} + \int_{(-\infty,-\log 2)}(1- 2\e^{\zz}) \mLevy_{\gamma} (d\zz), \quad k  = k_{\gamma}.
\end{equation}
\end{lem}
\begin{proof}
It is easy to check the {\it if} part by straightforward calculation. 
We now prove the {\it only if} part. If $\kappa =  \kappa_{\gamma}$, then the third order derivatives of $\kappa(q)$ and $\kappa_{\gamma}(q)$ are equal for every $q>2$, i.e. 
$$ \int_{(-\infty,0)} \left( \e^{q \xbar} \xbar^3 + \e^{q \zz} \zz^3 \right) \mLevy  (d\zz) = \int_{(-\infty,0)} \left( e^{q \xbar} \xbar^3 + \e^{q \zz}\zz^3 \right) \mLevy_{\gamma} (d\zz). $$
Therefore, for every $q>2$ there is 
$$\int_{(-\infty,0)} \e^{q\zz} \zz^3 \left( \mLevy (d\zz) + \bar{\mLevy}(d\zz) \right) =  \int_{(-\infty,0)} \e^{q\zz} \zz^3 \left( \mLevy_{\gamma} (d\zz) + \bar{\mLevy}_{\gamma}(d\zz) \right),$$
which implies that $\mLevy  + \mbar  = \mLevy_{\gamma} + \mbar_{\gamma}$. 
Iterating this argument over the lower order derivatives of $\kappa $ and $\kappa_{\gamma}$, we obtain the other identities in turn. 
\end{proof}

\subsection{Switching transformations and bifurcators}\label{sec:sw}
In order to give a construction of bifurcators of homogeneous cell processes, we now generalize the switching transformations between fragmenters in \cite{PitmanWinkel} to SNLPs.
Let $\xi$ be a SNLP with characteristics $(\sigma , c , \mLevy , k )$ and $p: (-\infty, 0) \to [0,1]$ be a measurable function, which will serve as switching probability, such that 
\begin{equation}\label{eq:st_H}
\int_{(-\infty,0)}   p(\zz) ~\mLevy ( d\zz)   < \infty.
\end{equation} 
We shall derive another SNLP $\Lp[\xi]$ from $\xi$ by switching according to $p$ in the following way. 
At each jump time $t> 0$ of $\xi$ with $\zz:= \Delta \xi(t)= \xi(t) - \xi(t-)< 0$, we {\it mark} this jump time with success probability $p(\zz)$ (so with failure probability $1-p(z)$ we do not mark it), independently of the other jumps. 
We thus define a point process by the marked jumps:
\begin{equation}
  \Delta_1 (t): =\begin{cases}
     \Delta \xi(t) & \text{if $t$ is a marked time},  \\
     0 & \text{otherwise}.
  \end{cases}
\end{equation}
Implicitly, the killing time $\zeta $ is never marked. We stress that the number of marked jump times is locally finite if and only if \eqref{eq:st_H} holds. Indeed, observing from the property of L\'evy processes (see e.g. \cite{Bertoin:Levy}) that $(\Delta \xi(t), t\geq 0)$ is a Poisson point process with characteristic measure $\mLevy $, we have that $\Delta_1$ is a Poisson point process with characteristic measure $\Lambda_1(d\zz):= p(\zz)\mLevy (d\zz)$. Next, we define a point process $\bar{\Delta}_1$ associated with $\Delta_1$ by 
\begin{equation}
  \bar{\Delta}_1 (t): =\begin{cases}
    \log (1 - \e^{\Delta_1(t)}) & \text{if } \Delta_1(t) \neq 0,  \\
     0 & \text{if } \Delta_1(t) = 0.
  \end{cases}
\end{equation}
Then $\bar{\Delta}_1$ is a Poisson point process with characteristic measure $\bar{\Lambda}_1(d\zz) := p(\xbar)\mbar (d\zz)$, where $\mbar $ is the image of $\mLevy $ by the map $\zz \mapsto \xbar$. Therefore, as \eqref{eq:st_H} holds, the processes
$$\xi_1(t) := \sum_{s\leq t} \Delta_1 (s)\quad \text{ and }\quad \bar{\xi}_1(t):= \sum_{s\leq t} \bar{\Delta}_1 (s)$$ 
are compound Poisson processes with respective (finite) L\'evy measures $\Lambda_1$ and $\bar{\Lambda}_1$. We finally define {\it the switching transform of $\xi$ according to $p$} by the process 
$$\Lp[\xi]:= \xi -  \xi_1 + \bar{\xi}_1.$$
\begin{lem}\label{lem:st}
Let $\xi$ be a SNLP with characteristics $(\sigma , c , \mLevy , k )$ and $p: (-\infty, 0) \to [0,1]$ be a measurable function that satisfies \eqref{eq:st_H}. Then the switching transform $\Lp$, derived from $\xi$ according to $p$, is a SNLP with characteristics 
 \begin{equation}\label{eq:st}\begin{cases} 
\Lp[\sigma] &:= \sigma, \\
 \Lp[\mLevy] (d\zz) &:=  (1- p(\zz)) \mLevy ( d\zz) + p(\xbar) \mbar (d\zz), \\
 \Lp[c]&:= c  + \int_{(-\infty,0)}  (1-2\e^{\zz}) p(\zz)\mLevy ( d\zz), \\
\Lp[k]&:= k .
\end{cases}
\end{equation}
Define $\Lp[\kappa]$ as in \eqref{eq:kappa} for $\Lp$, then $\Lp[\kappa] = \kappa$.  
 Further, $$\tb:= \inf \left\{ t\geq 0~:~  \xi(t) \neq \Lp[\xi](t) \right\}$$
has an exponential distribution with parameter $\int_{(-\infty, 0)\setminus\{-\log 2\}}   p(\zz) ~\mLevy ( d\zz)<\infty.$ Moreover, if $\tb< \infty$ then $\tb$ is a jump time of both $\xi$ and $\Lp[\xi]$ with   
$$ \exp (\xi(\tb))+ \exp(\Lp[\xi](\tb)) =\exp(\xi(\tb-)).$$ 
\end{lem}
\begin{proof}
 The L\'evy processes $(\xi -  \xi_1)$ and $\xi_1$ are independent since they never jump at the same time. For the same reason, the L\'evy processes $(\xi -  \xi_1)$ and $\bar{\xi}_1$ are also independent. Therefore, the Laplace exponent of $\Lp$ is $\Phi -\Phi_1 + \bar{\Phi}_1$, where $\Phi_1$ and $\bar{\Phi}_1$ are respective Laplace exponents of $\xi_1$ and $\bar{\xi}_1$. So we get \eqref{eq:st} and thus check that $\Lp[\kappa] = \kappa$ by straightforward calculation.
 
We next observe from the construction of $\Lp$ that 
$$ \inf \left\{ t\geq 0~:~  \xi(t) \neq \Lp[\xi](t) \right\}= \inf \left\{ t\geq 0~:~ \Delta_1(t) \neq 0 \text{ and } \Delta_1(t) \neq -\log 2 \right\},$$
which implies the second part of the statement. 
\end{proof}

\begin{rem}
It follows from \eqref{eq:Levym} that for every $a\geq 2$ the function $z \mapsto (1- \e^{z})^{a}$ satisfies \eqref{eq:st_H}. 
However, the function $\zz \mapsto (1- \e^{\zz})$, which would correspond to the size-biased pick between $\exp(\Delta \xi (t))$ and $(1- \exp(\Delta \xi (t)))$ (see Section 2.2 in \cite{PitmanWinkel}) cannot satisfy \eqref{eq:st_H} unless  $ \int_{(-\infty,0)} (|\zz| \wedge 1) \mLevy  (d\zz) < \infty$.
\end{rem}

\begin{lem}\label{lem:st1} 
If $\kappa_{\gamma} = \kappa$, then for every measurable function $p \colon (-\infty,0) \to [0,1]$ such that
\begin{equation}\label{eq:pbar}
\int_{(-\infty,0)} p( \zz) \mLevy(d \zz)< \infty \qquad\text{ and } \qquad p(\zz) + p(\xbar)=1  \text{ for every }z\in (-\infty,0),
\end{equation}
there is $\int_{(-\infty,0)} p( \zz) \mLevy_{\gamma}(d \zz)< \infty$ and $\Lp[\gamma]\eqdis\Lp[\xi]$. 
\end{lem}
The function $\zz \mapsto \ind{\zz<-\log 2} + \frac{1}{2}\ind{\zz=-\log 2}$ gives an example that satisfies \eqref{eq:pbar}. 
\begin{proof}[Proof of Lemma \ref{lem:st1}] 
As $\kappa_{\gamma} = \kappa$, it follows from Lemma \ref{lem:triple} and \eqref{eq:Levym} that $\mLevy_{\gamma} -\mLevy $ is a finite signed measure and hence we have
$$\int_{(-\infty,0)} p( \zz) \mLevy_{\gamma}(d \zz)\leq \int_{(-\infty,0)} p( \zz) \mLevy (d \zz) +  \int_{(-\infty,0)} |\mLevy_{\gamma} -\mLevy |(d\zz) < \infty.$$
So the switching transforms $\Lp[\gamma]$ and $\Lp[\xi]$ are well-defined. As \eqref{eq:pbar} holds, by combining Lemma \ref{lem:triple} and Lemma \ref{lem:st}, we get that the characteristics of $\Lp[\gamma]$ are the same as those of $\Lp[\xi]$.   
\end{proof}
We next see that the SNLPs that have the same $\kappa$ are related to each other via the switching transformations. 
\begin{prop}\label{prop:st}
If $\kappa_{\gamma} = \kappa$, then $\gamma\eqdis \Lp[\xi]$, where $p$ is the measurable function defined by Radon-Nikodym derivative
$$p  (\zz):=  \mbar_{\gamma} (d\zz) / (\mLevy_{\gamma} (d\zz)+ \mbar_{\gamma} (d\zz)).$$
\end{prop}

\begin{proof}
Observe that
\begin{linenomath}\begin{align}
 \int_{(-\infty,0)} p ( \zz) \mLevy_{\gamma} (d \zz)& \leq   \int_{(-\infty,-\log 2)}  \mLevy_{\gamma} (d\zz) + \int_{(-\log 2,0)} \mbar_{\gamma}  (d\zz)< \infty,
\end{align}\end{linenomath}
then the switching transform $\Lp[\gamma]$ is well-defined, and we deduce from Lemma \ref{lem:st} that $\gamma^{[p ]}\eqdis \gamma$. Note that $p (\zz) + p (\xbar)=1$ for every $z\in (-\infty,0)$, then it follows from Lemma \ref{lem:st1} that $\Lp$ is also well-defined and $ \Lp\eqdis \gamma^{[p ]}$. So we conclude that $\gamma\eqdis \Lp[\xi]$.  
\end{proof}

We finally present a construction of a bifurcator of homogeneous cell processes, which has the following precise definition.  
\begin{defn}\label{defn:bifH}
A pair of homogeneous cell processes $(X ,Y )$ is a {\bf bifurcator} if it satisfies the following properties: 
  \begin{enumerate}[label=(\roman*)]
  \item Let $\tb:= \inf \{t\geq 0: X (t) \neq \tildeX  (t) \}$. There is almost surely either $\tb=\infty$ or the identity
    \begin{equation}
      X (\tb)+ \tildeX (\tb) = X (\tb-)= \tildeX (\tb-).
    \end{equation}
  \item  (Asymmetric Markov branching property)
Conditionally given $\tb>t$, the pair $(X (r)/ X (t), \tildeX (r)/\tildeX (t))_{r\geq t} $ is a copy of $(X ,Y )$; conditionally given $\tb\leq t$, the two processes $(X (r)/ X (t))_{r\geq t}$ and $(\tildeX (r)/\tildeX (t))_{r\geq t} $ are independent copies of $X $ and $Y $ respectively. 
  \end{enumerate}
\end{defn}
This definition generalizes bifurcators of fragmenters in \cite{PitmanWinkel}. 
We shall later extend this notion to general cell processes, see Definition \ref{defn:bif}. 

\begin{lem}\label{lem:bifH}
  If $\kappa=\kappa_{\gamma}$, then there exists a bifurcator of homogeneous processes $(X,Y)$, such that the marginal laws of $X$ and $Y$ are the laws of $\exp(\xi)$ and $\exp(\gamma)$ respectively. 
\end{lem}

\begin{proof} 
Since $\kappa=\kappa_{\gamma}$, we can build as in Proposition \ref{prop:st} the switching transform $\xi^{[p]}$ derived from $\xi$ such that $\xi^{[p]}\eqdis \gamma$. We stress that $\xi$ and $\xi^{[p]}$ are still coupled after the switching time $\tb:= \inf\{t\geq 0:~ \xi(t)\neq \xi^{[p]}(t)\}$. However, let us define a process $Y$ by
$$Y(t) :=\ind{t< \tb} \exp(\xi(t)) + \ind{t\geq \tb} \exp(\xi^{[p]}(\tb) + \gamma'(t-\tb) )  , \quad t\geq 0,$$
where $\gamma'$ is a copy of $\gamma$, independent of $\xi^{[p]}$ and $\xi$.
Then we easily check that $Y \eqdis \exp(\gamma)$ and the pair of homogeneous processes $(X:=\exp(\xi), Y)$ satisfies Definition \ref{defn:bifH}. 
\end{proof}



\subsection{Binary branching L\'evy processes}\label{sec:BLP}
Let $\xi_b$ be a SNLP with characteristics $(\sigma_b, c_b, \mLevy_b,k_b)$ and  $\mmu_b$ be a L\'evy measure with support on $[-\log 2, 0]$ that satisfies 
\begin{equation}\label{eq:mu}
\int_{\Rdtwo} (1 \wedge \zz^2) \mmu_b(d \rr)<\infty.
	\end{equation}  
Informally speaking, a {\it binary branching L\'evy process (BBLP)} introduced in \cite{Bertoin:CF} models the evolution of a particle system, in which each particle moves in $\R$ according to the SNLP $\xi_b$, independently of the other particles, and at rate $\mmu_b(d\zz)$ each particle gives birth to two children scattered on $\R$, whose initial positions relative to the position of the parent at death are given by $\zz$ and $\xbar = \log(1-\e^{\zz})$. 
We further add a properly chosen positive drift for the entire system, which is an analogue of the compensation term in the L\'evy-Khintchine formula \eqref{eq:LK}, so that the particles in this system do not all shift to $-\infty$ instantaneously. Proposition 3 in \cite{Bertoin:growth} establishes a close connection between BBLPs and homogeneous growth-fragmentations. We will extend this connection in the next subsection. Before that, we recall some basic facts of BBLPs in this subsection.

Let us represent the formal construction of BBLPs in \cite{Bertoin:CF}, starting with the case when the branching occurs with a finite intensity, i.e. $\mmu_b(\Rdtwo)< \infty$.
Write $\Utwo:= \bigcup_{n=0}^{\infty}\{\vl,\vr \}^n$ for the binary Ulam-Harris tree with $\{\vl, \vr \}^0:= \emptyset$ by convention, so for every $i\in\N$, an element in $\{\vl,\vr \}^i$ is a word $v= (n_1,n_2, \ldots, n_i)$ composed of $i$ letters of the alphabet $\{\vl,\vr \}$. We write $|v|:=i$ for the generation of $v$ and $(v\vl,v\vr)$ for its children, where $v\vl$ would be referred to as the {\it left} child and $v\vr$ as the {\it right} child. For every $j\leq |v|$, we denote by $[v]_j:=(n_1,n_2, \ldots,n_j)$ the ancestor of $v$ at the $j$-th generation. 

\begin{defn}\label{defn:BLP}
Let $\xi_b$ be a SNLP with characteristics $(\sigma_b, c_b, \mLevy_b,k_b)$ and $\mmu_b$ be a finite measure on $\Rdtwo$. We consider three independent processes $(\Deltab_v)_{v \in \Utwo}$, $(L_v)_{v \in \Utwo}$ and $(\Deltaa_v)_{v \in \Utwo}$ such that: 
\begin{itemize}
\item $(\Deltab_v)_{v \in \Utwo}$ is a family of i.i.d. exponential variables with parameter $\mmu_b(\Rdtwo)$. 
\item $(L_v)_{v \in \Utwo}$ is a family of independent SNLP distributed as 
$$\xi_b(t) + \left(\int_{\Rdtwo} (1-\e^{\zz})\mmu_b(d\rr)\right)t, \quad t\geq 0.$$  
\item $(\Deltaa_{v\vl}, \Deltaa_{v\vr})_{v \in \Utwo}$ is a family of i.i.d. random variables, such that 
$\Deltaa_{v\vl}$ is distributed according to the conditional probability $\mmu_b (\cdot ~|~ \Rdtwo) $ and $\Deltaa_{v\vr} = \overline{\Deltaa_{v\vl}} =\log( 1 - \exp(\Deltaa_{v\vl}))\leq \Deltaa_{v\vl}$.  
\end{itemize}
Define for every $v\in \Utwo$ the birth time by $\beta_v := \sum_{j=0}^{|v|-1}\Deltab_{[v]_j}$, and iteratively the positions of its children at birth by $(a_{vi} = a_v+ L_v(\Deltab_v) + \Deltaa_{vi}, i\in \{\vl,\vr\})$,  with $a_{\emptyset}=0$. We agree that $L_v(s)=-\infty$ if $L_v$ is killed before $s$. Then the positions of the particles alive at time $t\geq 0$ form a multiset of elements in $\R$ (which is a generalization of the concept of a set that, unlike a set, allows multiple instances of the multiset's elements)
$$\Zs (t):= \multiset{a_v + L_v(t-b_v)~:~ v\in \Utwo, \beta_v \leq t < \beta_v+ \Deltab_v}.$$
The process $(\Zs (t), t\geq 0)$ is a {\bf binary branching L\'evy process (BBLP) with characteristics $(\sigma_b, c_b, \mLevy_b,k_b,\mmu_b)$}. 
\end{defn}
\begin{rem}
  A multiset $\ms$ could be equivalently viewed as the point measure $\sum_{i\in \ms}\delta_{i}$, where $\delta$ stands for the Dirac mass. So we can identify $\Zs$ with a point process.   
\end{rem}

We next extend the construction to infinite branching intensify. Suppose that $\mmu_b(\Rdtwo)= \infty$. 
For every $\levb \leq -\log 2$, let us set
\begin{equation}\label{eq:cutb}
\cutb{\mmu}_b:=\ind{[-\log 2, \lev)}\mmu_b, \quad \cutb{\mLevy}_b:=\mLevy_b + \ind{[\lev,0)}\mmu_b.
\end{equation}
We know from Lemma 3 in \cite{Bertoin:CF} that we can construct a family of processes $(\Zs^{\levb}, -\infty <\levb\leq -\log 2)$ in the same probability space, with each $\Zs^{\levb}$ a BBLP with characteristics $(\sigma_b, c_b,\cutb{\mLevy}_b,k_b, \cutb{\mmu}_b)$ in the sense of Definition \ref{defn:BLP} (we stress that \eqref{eq:mu} assures that $\cutb{\mmu}_b$ is a finite measure), such that for every $\levb\leq \levb'\leq -\log 2$ there is $\cutb[\levb']{(\Zs^{\levb})} = \Zs^{\levb'}$, where $\cutb[\levb']{(\Zs^{\levb})}$ is the system derived from $\Zs^{\levb}$ by keeping at each branching event the child particle that is closer to the mother, and suppressing the other child particle (together with its offspring) whenever it is born at distance from its mother $\geq |\levb'|$. 


\begin{defn}\label{defn:BLP1}
In the notation above, suppose that $\mmu_b$ is a L\'evy  measure on $\Rdtwo$ that verifies \eqref{eq:mu}. Then the limit process (by monotonicity in the sense of multiset inclusion)
$$\Zs(t) := \lim_{\levb \to -\infty}\uparrow \Zs^{\levb}(t) , \quad t\geq 0$$ 
is a {\bf BBLP with characteristics $(\sigma_b, c_b ,\mLevy_b, k_b, \mmu_b)$}. 
\end{defn}

\begin{rem} Our notation is slightly different from that of \cite{Bertoin:CF}. In the sense of Definition 2 in \cite{Bertoin:CF}, a BBLP with characteristics $(\sigma_b, c_b,\mLevy_b, k_b, \mmu_b)$ is characterized by $(\sigma_b, c_b-k_b, \mu_b)$, where $\mu_b$ is a measure on the space  
	 $$\left\{ (r_1, r_2, -\infty,\ldots ,-\infty) ~:~ \e^{r_1}+ \e^{r_2} \leq  1, 0> r_1\geq r_2\geq -\infty \right\}, $$
	and $ $ is given by the sum of the following three measures: the image of $\mLevy_b$ by the map $\zz \mapsto (\zz, -\infty,\ldots ,-\infty)$, the image of $\mmu_b$ by the map $\zz \mapsto (\zz, \xbar,-\infty,\ldots ,-\infty)$ and $k_b \delta_{(-\infty,\ldots ,-\infty)}$.
\end{rem}

Let $\Phi_b$ be the Laplace exponent of the SNLP $\xi_b$ with characteristics $(\sigma_b, c_b,\mLevy_b, k_b)$. Introduce $\kappa_b: [0,\infty) \to (-\infty, \infty]$ by
\begin{equation}\label{eq:kappaCFP}
\kappa_b (q) :=\Phi_b(q) +\int_{\Rdtwo}  \left(\e^{q \zz} + (1-\e^{\zz})^{q} -1 +  q(1-\e^{\zz}) \right) \mmu_b(d\rr), \quad q \geq 0, 
\end{equation}
then $\kappa_b$ serves as cumulant for the BBLP $\Zs$. 
Specifically, we know from Theorem 1 in \cite{Bertoin:CF} that for every $q\geq 2$, there is $\kappa_b (q)< \infty$ and
\begin{equation}\label{eq:mart}
  \Exp{\sum_{z\in \Zs(t)} \e^{q \zz}} = \e^{\kappa_b (q) t}\quad \text { for all } t\geq 0.
\end{equation}
We now check that if $\mLevy_b = 0$, then the cumulant determines the distribution of the BBLP in the following sense.

 \begin{lem}\label{lem:CFP}
Let  $\Zs$ and $\Zs'$ be two BBLPs with respective characteristics $(\sigma_b, c_b ,\mLevy_b , k_b, \mmu_b)$ and $(\sigma'_b, c'_b ,\mLevy'_b,$ $k'_b, \mmu'_b)$. If $\mLevy_b = \mLevy'_b=0$ and their cumulants 
$\kappa_b = \kappa'_b$, then $\Zs$ and $\Zs'$ have the same law. 
 \end{lem}

 \begin{proof}
   Since the third order derivatives of $\kappa'_b$ and $\kappa_b$ are equal for all $q>2$, by a similar argument as in the proof of Lemma \ref{lem:triple}, we find that $\mmu_b' + \bar{\mmu}_b'= \mmu_b + \bar{\mmu}_b$. As $\mmu_b'$ and $\mmu_b$ are supported on $[\log 2, 0]$, we hence find that $\mmu'_b=\mmu_b$. By iterating this argument over the lower order of derivatives, we conclude that $\Zs$ and $\Zs'$ have the same 
characteristics, thus the same law. 
 \end{proof}

\subsection{Homogeneous growth-fragmentations}\label{sec:HGF}
For every $x>0$, write $P_x$ for the law of the homogeneous cell process $X: =x \exp(\xi)$, where $\xi$ is a SNLP with $\kappa$ defined as in \eqref{eq:kappa}. If $\xi$ is killed at a time $\zeta$, then by convention we denote $X(t) = \partial$ for all $t\geq \zeta$, where $\partial$ is the cemetery state. 
Let $\Xs$ be a homogeneous growth-fragmentation associated with $X$, which was informally described in the \hyperref[sec:intro]{Introduction}. By connecting to branching L\'evy processes, we shall prove in this section that the law of $\Xs$ is characterized by the cumulant function $\kappa$. 


In that direction, let us present the rigorous construction of $\Xs$, which is only a slight modification of that in \cite{Bertoin:growth}. We start with listing the jumps of $X$ in the following way. Fix $q>2$ and $K > \kappa(q)$. Recalling that the jump process $\Delta \xi$ is a Poisson point process with characteristic measure $\mLevy$ and using the compensation formula (see e.g. \cite{Bertoin:Levy}), we get for every $x>0$
\begin{linenomath}\begin{align} 
  \Exp[x]{ \sum_{0\leq s}|\Delta X(s)|^q \e^{-K s}} & = \Exp[x]{\sum_{0\leq s} X(s-)^q (1 - \e^ {\Delta \xi(s)})^q \e^{-K s}} \\
&= \Exp[x]{\int_{0}^{\infty} \e^{-K s}X(s-)^q ds \int_{(-\infty,0)} (1 - \e^{z})^q \mLevy (dz)} \\
&= \frac{\kappa(q) -\Phi(q)}{K -\Phi(q)}  x^q , \label{eq:hom}
  \end{align}\end{linenomath}
where $E_{x}$ stands for mathematical expectation under $\muA_{x}$. This implies that $P_x$-almost surely 
$$\sum_{s\geq 0}|\Delta X(s)|^q \e^{-K s}<\infty.$$ 
We may therefore list the jump times of $X$ in a sequence $(t_i, i\in \N)$ such that $(|\Delta X(t_i)|^q \e^{-K {t_i}}, i\in \N))$ is decreasing. By convention, if $X$ has a finite number of jumps, then the tail of this sequence is filled with $\infty$ with $\Delta X(\infty)=\partial$. In the sequel, {\bf the $i$-th jump time of $X$} shall always refer to the $i$-th element $t_i$ in this sequence. 

Let us give some basic notations. Let $\U:= \bigcup_{i=0}^{\infty} \N^{i}$ be the Ulam-Harris tree, by convention $\N^{0} = \{\emptyset\}$. 
An element $u\in \U$ is a finite sequence of natural numbers $u = (n_1, \ldots ,n_{|u|})$ where $|u|\in \N$ stands for the generation of $u$. We write $u_- = (n_1, \ldots ,n_{|u|-1})$ for its mother and $uk = (n_1, \ldots n_{|u|} ,k)$ for its $k$-th daughter with $k\in \N$. We also denote ${[u]}_i=(n_1, \ldots ,n_{i})$ for every $i \leq |u|$ with ${[u]}_0 = \emptyset$ by convention.  

We next construct the {\bf cell system driven by $X$}, which is a family of homogeneous cell processes indexed by $\U$
$$\Xc:= (\Xc_u, ~ u\in \U),$$ 
where each $\Xc_u$ depicts the evolution of the size of the cell indexed by $u$ as time passes. Specifically, we fix an arbitrary $x>0$, which is the initial size of the ancestor cell. Then we set the birth time of $\emptyset$ at $b_{\emptyset}:= 0$ and let the life career $\Xc_{\emptyset} = (\Xc_{\emptyset}(t), t\geq 0)$ be a process of law $\muA_{x}$. Given the life path of $\Xc_{\emptyset}$, then we generate the first generation. For $i\in \N$, say the $i$-th jump time of $\Xc_{\emptyset}$ is $t_i$ and $x_i:= -\Delta \Xc_{\emptyset}(t_i)$, we then set $b_i =t_i$ and build a sequence of conditional independent processes $(\Xc_i)_{i\in \N}$ with respective conditional distribution $\muA_{x_i}$. By convention, if $t_i= \infty$ (which means that $\Xc_{\emptyset}$ has less than $i$ jumps), then we agree that the cell $i$ as well as all its progeny have degenerate life careers, i.e. for every $v\in \U$ we set $\Xc_{iv} \equiv \partial$ and $b_{iv}= \infty$. We continue in this way to construct higher generations recursively. 
Write $\Pc_x$ for the law of this cell system $\Xc$ (recall that $x>0$ indicates the initial size of the Eve $\emptyset$, i.e. $\Xc_{\emptyset}(0)= x$).
According to \cite{Jagers}, the probability distribution $\Pc_x$ indeed exists and is uniquely determined by the above description. 

Finally, for every $t\geq 0$ let $ \Xs(t)$ be the multiset whose elements are sizes of the cells alive at time $t$, i.e. $$ \Xs(t) := \multiset{\Xc_u(t-b_u)~:~ u\in \U, b_u\leq t},$$
then we refer to $\Xs = (\Xs(t), t\geq 0)$ as a {\bf growth-fragmentation process driven by $X$} and write $\Ps_{x}$ for the law of $\Xs$ under $\Pc_{x}$. 

\begin{rem}
The construction of the cell system $\Xc$ is only a slight modification of that of a cell system in \cite{Bertoin:growth}, and that of a general branching process (also called Crump-Mode-Jagers process) in \cite{Jagers}. The only difference lies in the fact that, in \cite{Bertoin:growth} daughters are listed in decreasing order of the sizes at birth, and in \cite{Jagers} daughters are enumerated by their birth times. However, in whole generality, it is not always possible to enumerate the jumps of a homogeneous process $X$ in decreasing order of jump sizes or increasing order of jump times. 
\end{rem}

\begin{rem} 
If we use a different way to enumerate the jumps of $X$, it is intuitively clear that the new cell system is the same as the original one, up to a permutation of $\U$. Thus the growth-fragmentation $\Xs$ obviously does not depend on the method of enumeration and the law of $\Xs$ is determined by $X$. 
\end{rem}

We now present a connection between homogeneous growth-fragmentation processes and BBLPs.
\begin{prop}\label{prop:1}
Let $\xi$ be a SNLP with characteristics $(\sigma, c, \mLevy, k)$ and $\kappa$ defined as in \eqref{eq:kappa} and $\Xs$ be a homogeneous growth-fragmentation process (starting from $1$) driven by $X:= \exp(\xi)$. Then the process $\log \Xs$ is the unique (in law) BBLP with cumulant $\kappa$ and $\mLevy_b = 0$. Specifically, $\log \Xs$ has characteristics $(\sigma, c_b, 0, k, \mmu_b)$, where 
\begin{equation}\label{eq:BLP}
 c_b= c + \int_{(-\infty,-\log 2)}(1- 2\e^{\zz}) \mLevy (d\zz), \text{~and}\quad \mmu_b= \ind{(-\log 2, 0)}(\mLevy+\mbar) + \frac{1}{2}\ind{ -\log 2}(\mLevy+\mbar).
\end{equation}
In particular, we have 
$$\Exps{\sum_{x\in \Xs(t)} x^q} = \exp\left(\kappa(q) t\right) \quad \text{for all } q\geq 2 \text{ and } t\geq 0.$$ 
\end{prop}

\begin{rem}
So the homogeneous growth-fragmentation $\Xs$ is a {\it compensated fragmentation process} in the sense of \cite{Bertoin:CF}. When $\sigma= 0$, $c=0$ and $ \int_{(-\infty,0)} (1-\e^{\zz}) \mLevy(d\zz) <\infty$, it is a homogeneous fragmentation process in the sense of \cite{Bertoin:Homo}. 
\end{rem}

Proposition \ref{prop:1} extends Proposition 3 in \cite{Bertoin:growth}, which obtained the same result for the case when the L\'evy measure $\mLevy$ of $\xi$ satisfies $\mLevy( (-\infty, -\log 2)) = 0$.
Before tackling the proof of Proposition \ref{prop:1}, let us provide a variation of Theorem \ref{thm:2'}, which summarizes the discussion in this section. 

\begin{cor}\label{cor:prop1}
Let $\xi$ and $\tilde{\xi}$ be two SNLPs with respective cumulant functions $\kappa$ and $\tilde{\kappa}$ defined as in \eqref{eq:kappa}. Let $\Xs$ and $\tilde{\Xs}$ be the homogeneous growth-fragmentations associated with $\xi$ and $\tilde{X}$ respectively (with the same initial size of ancestor $x>0$),
The following statements are equivalent: 
  \begin{enumerate}[label=(\roman*)]
  \item $\kappa = \tilde{\kappa}$;
  \item  $\tilde{\xi}$ has the same law as a switching transform of $\xi$;
  \item the homogeneous growth-fragmentations $\Xs$ and $\tilde{\Xs}$ have the same finite-dimensional distribution.
  \end{enumerate}
\end{cor}

\begin{proof}[Proof of Corollary \ref{cor:prop1}]
 $(i)\Leftrightarrow (ii)$: The two directions follow respectively from Proposition \ref{prop:st} and Lemma \ref{lem:st}.

$(i)\Leftrightarrow (iii)$: We know from Proposition \ref{prop:1} that $\log \Xs$ and $\log \tilde{\Xs}$ are BBLPs with respective cumulants $\kappa$ and $\tilde{\kappa}$. 
If $\Xs$ and $\tilde{\Xs}$ have the same finite-dimensional distribution, then so do the BBLPs $\log \Xs$ and $\log \tilde{\Xs}$, and in particular their cumulant are the same. 
Conversely, if $\kappa = \tilde{\kappa}$, then we deduce from Lemma \ref{lem:triple} or Lemma \ref{lem:CFP} that the BBLPs $\log \Xs$ and $\log \tilde{\Xs}$ have the same characteristics, thus the same finite-dimensional distribution. 
\end{proof}

The rest of this section is devoted to the proof of Proposition \ref{prop:1}. 
\begin{proof}[Proof of Proposition \ref{prop:1}]

The idea of the proof is similar to that of Proposition 3 in \cite{Bertoin:growth}.
Let $\Zs$ be a BBLP with characteristics $(\sigma, c_b,0, k, \mmu_b)$ and write $(\Zs^{\levb}, -\infty < \levb \leq -\log 2)$ for the family of BBLPs as in Definition \ref{defn:BLP}, each $\Zs^{\levb}$ a BBLP with characteristics $(\sigma, c_b,\ind{[\lev,0)}\mmu_b, k,\ind{[-\log 2, \lev)}\mmu_b)$, 
such that 
$$ \Zs(t) = \lim_{\levb \to -\infty}\uparrow \Zs^{\levb}(t), \quad t \geq 0. $$
We shall check for every $\levb\in (-\infty , -\log 2)$ that $\exp(\Zs^{\levb})$ has the same dynamics as a truncated cell system associated with the cell process $X = \exp(\xi)$, in which each cell $u\in \U$ is killed at the first instant $s$ with $\Xc_{u}(s)\leq \e^{\levb} \Xc_{u}(s-)$, together with her future descents (born at time $>s$); furthermore, for each $j\in \N$ the daughter cell $uj$ is killed at birth (together with its descents) whenever her size is less than or equal to $\e^{\levb}$ times the size of her mother immediately before the birth event, i.e. $\Xc_{uj}(0)\leq \e^{\levb} \Xc_{u}(b_{uj}-)$. 
Letting $\levb \to -\infty$, we conclude from Definition \ref{defn:BLP} and the monotonicity that $\log \Xs$ has the same distribution as $\Zs$. Then it is straightforward to check that $\log \Xs$ indeed has cumulant $\kappa$ and the identity in the proposition thus follows from \eqref{eq:mart}. The uniqueness of $\log \Xs$ follows from Lemma \ref{lem:CFP}.  

So it remains to prove that $\exp(\Zs^{\levb})$ indeed has the same law as the truncated cell system. 
In this direction, let us construct an auxiliary particle system as follows, which is a minor modification of Definition \ref{defn:BLP}. Fix an arbitrary $\levb< -\log 2$. 
Let us consider three independent sequences of processes $(\Deltab_v)_{v \in \Utwo}$, $(L_v)_{v \in \Utwo}$ and $(\Deltaa_{v\vl}, \Deltaa_{v\vr})_{v \in \Utwo}$ such that: 
\begin{itemize}
\item $(\Deltab_v)_{v \in \Utwo}$ is a family of i.i.d. exponential variables with parameter $\Lambda((\infty, \lev))$;\item $(L_v)_{v \in \Utwo}$ is a family of independent copies of SNLP $\tilde{\xi}$ with characteristics $(\sigma, \tilde{c}, \ind{[\lev, 0)}\Lambda, k)$ where $\tilde{c}:= c + \int_{(-\infty, \lev)}(1-\e^{\zz})\Lambda(d\zz)$.
\item $(\Deltaa_{v\vl}, \Deltaa_{v\vr})_{v \in \Utwo}$ is a family of i.i.d. pairs of random variables such that each $\Deltaa_{v\vl}$ is distributed according to the conditional probability $\Lambda(\cdot ~|~(-\infty, \lev))$ and $\Deltaa_{v\vr} = \overline{\Deltaa_{v\vl}} = \log (1- \exp(\Deltaa_{v\vr}))$. 
\end{itemize}
Write $\beta_v := \sum_{j=0}^{|v|-1}\Deltab_{[v]_j}$ for the birth time, and define by induction $a_{vi} = a_v+ L_v(\Deltab_v)+ \Deltaa_{vi}$ for $i\in \{\vl, \vr\}$ with $a_{\emptyset}=0$. So we define $\Ls$ by
 $$\Ls (t):= \multiset{a_v + L_v(t-\beta_v) ~:~ v\in \Utwo, \beta_v \leq t < \beta_v+  \Deltab_v}, \quad t\geq 0.$$

We stress that unlike in Definition \ref{defn:BLP}, $\ind{(-\infty, \lev)}\Lambda$ is not supported on $[-\log 2, 0)$, so $\Deltaa_{v\vl}$ may be possibly smaller than $\Deltaa_{v\vr}$. 
However, we may obtain a BBLP by changing the indices of the particles. Specifically, let us define a bijection $h: \Utwo \to \Utwo$ in the following way. Let $h(\emptyset):=\emptyset$. Given $h(v)$ with $v\in \Utwo$ by induction, then we assign the index of $\max(\Deltaa_{h(v)\vl} , \Deltaa_{h(v)\vr} )$ to $h(v\vl)$ and let $h(v\vr)$ be the sister of $h(v\vl)$. 
We therefore define $(\Deltaa'_{v\vl},\Deltaa'_{v\vr}):= (\Deltaa'_{h(v\vl)},\Deltaa'_{h(v\vr)})$, $\beta'_{v}:= \beta_{h(v)}$ and $L'_{v}:= L_{h(v)}$ for each $v\in \Utwo$, and further define recursively $a'_{vi} := a'_v+ L'_v(\Deltab'_v)+ \Deltaa'_{vi}$. As $h$ is a bijection, it is plain that
    $$\Ls (t)= \multiset{a'_v + L'_v(t-\beta'_v)~:~ v\in \Utwo,\beta'_v \leq t < \beta'_v+ \Deltab'_v}, \quad t\geq 0.$$
Let $$\mmu_{\Ls}=\frac{1}{2}\ind{-\log 2}(\Lambda +\mbar)+ \ind{(-\log 2, \lev)}\Lambda+\ind{( -\log 2,0)}\mbar,$$
then $\mmu_{\Ls}$ is supported on $[-\log 2, 0)$ and we observe that $((\Deltaa'_{h(v)\vl} , \Deltaa'_{h(v)\vr} ), v\in \Utwo)$ is a family of i.i.d. random variables such that $\Deltaa'_{h(v)\vl}$ has conditional law $\mmu_{\Ls}(\cdot~|~\Rdtwo)$ and $\Deltaa'_{h(v)\vr} = \overline{\Deltaa'_{h(v)\vl}}$, that $(\beta'_{v}, v \in \Utwo)$ is a family of i.i.d. exponential variables with parameter $\Lambda((-\infty, \lev))=\mmu_{\Ls}(\Rdtwo)$ and that $(L'_{v}, v \in \Utwo)$ is a family of independent copies of $\tilde{\xi}$. 
Using this point of view, we hence deduce that $\Ls$ is a BBLP as in Definition \ref{defn:BLP}, with characteristics $(\sigma, c_{b} ,\mLevy_{\Ls}:=\ind{[\lev, 0)}\Lambda, k, \mmu_{\Ls})$, where we have used the fact that
$$ \tilde{c} - \int_{[-\log 2, 0)} (1-\e^{\zz})\mmu_{\Ls}(d \zz) =c +  \int_{(-\infty, -\log 2)}(1- 2\e^{\zz})\Lambda(d\zz)= c_{b}. $$
 
Let us next give some remarks on the {\it leftmost} branch of the particle system $\Ls$, that is the process obtained by concatenating the segments of size processes of particles $\left\{\emptyset, \vl, \vl\vl, \vl\vl\vl, \ldots \right\}=:\Rm  \subset \Utwo$:  
\begin{equation}\label{eq:Rm}
\A_{\Rm}(t)  := \sum_{v\in \Rm} \ind{\beta_v \leq t < \beta_v+ \Deltab_v} (a_v + L_v(t-\beta_v)), \quad t\geq 0.
\end{equation}
Using elementary properties of L\'evy processes, we find that $\A_{\Rm}$ has the same distribution as $\xi$. We also notice that for every time $t\geq 0$ when $\Delta \A_{\Rm}(t) < \lev$, that is equivalently $\exp(\A_{\Rm})$ has a jump of size $-\Delta \exp(\A_{\Rm}(t))>\e^{\levb}\exp(\A_{\Rm}(t-))$, there is $t = \beta_v + \Deltab_{v}$ for a certain $v\in \Rm$. A fortiori, for every $s\geq 0$ such that $\Delta \A_{\Rm}(s) <\levb< \lev$, that is equivalently $\exp(\A_{\Rm}(s))\leq \e^{\levb}\exp( \A_{\Rm}(s-))$, there is $s = \beta_w + \Deltab_{w}$ for a certain $w\in \Rm$.

We finally consider the process $\hat{\Ls}$, which is associated with a system derived from $\Ls$, by suppressing for each $v\in \Utwo$ the child that corresponds to $\Deltaa_{v\vl}$ whenever $\Deltaa_{v\vl} \leq \levb$. So we can explain the dynamics of $\exp(\hat{\Ls})$ as follows. This system starts with an Eve cell whose size evolves according to $\Xc_{\emptyset}:= \exp(\A_{\Rm})$, and the Eve cell is killed (together with her future descents) at the first instant $s\geq 0$ when there is $\Xc_{\emptyset}(s)\leq \e^{\levb} \Xc_{\emptyset}(s-)$. Further, for each time $t\leq s$ when $\Xc_{\emptyset}$ has a jump of size $y := -\Delta \Xc_{\emptyset}(t)>\e^{\levb} \Xc_{\emptyset}(t-)$, there is $t = \beta_v + \Deltab_{v}$ for a certain $v\in \Rm$, then a daughter cell with initial size $y$ is born and the size of this daughter cell evolves according to the process $\exp(\A_{v\vr \Rm})$, where $\A_{v\vr \Rm}$ is the process associated with $v\vr\Rm:= \left\{ v\vr w ~:~w\in \Rm \right\}$ as in \eqref{eq:Rm}. Note that the process $\exp(\A_{v\vr \Rm})$ has the same distribution as $-y \exp(\xi)$. This daughter cell evolves independently of the other daughter cells, is killed at the first instant when her size drops suddenly by factor smaller than $\e^{\levb}$, and gives birth to grand-daughter cells each  time her size drops suddenly by factor smaller than $\e^{\lev}$ (note that $\e^{\lev}>\e^{\levb}$) before being killed (with killing time included). We continue so on and so forth to obtain the higher generations. So we conclude that $\exp(\hat{\Ls})$ indeed has the same law as a truncated cell system associated with $X = \exp(\xi)$. 

On the other hand, using the point of view that $\Ls$ is a BBLP with characteristics $(\sigma, c_{b},\mLevy_{\Ls} ,k, \mmu_{\Ls} )$, since $\Deltaa_{v\vr}> \levb$ always holds by the construction, we may equivalently view $\hat{\Ls}$ as the system obtained from $\Ls$ by suppressing for each $v\in \Utwo$ the smaller child $\Deltaa'_{v\vr}$ whenever $\Deltaa'_{v\vr} \leq \levb$. We hence deduce from Lemma 3 in \cite{Bertoin:CF} that $\hat{\Ls}$ is a BBLP with characteristics $(\sigma, c_{b},\cutb{\mLevy}_{\Ls} ,k, \cutb{\mmu}_{\Ls} )$, where $\cutb{\mLevy}_{\Ls}$ and $\cutb{\mmu}_{\Ls}$ are derived from $\mLevy_{\Ls}$ and $\mmu_{\Ls}$ as in \eqref{eq:cutb}. We check that $(\cutb{\mLevy}_{\Ls} , \cutb{\mmu}_{\Ls} ) =(\ind{[\lev,0)}\mmu_b,\ind{[-\log 2, \lev)}\mmu_b)$, so the two BBLPs $\hat{\Ls}$ and $\Zs^{\levb}$ have the same characteristics, which ends the proof.   
\end{proof}

 
\section{Markovian growth-fragmentation processes and bifurcators}\label{sec:cell}
In this section, we shall extend the notion of bifurcator to general cell processes and further establish a sufficient condition for different Markovian growth-fragmentations to have the same distribution, which finally orients us toward the proofs of Theorem \ref{thm:2'} and Theorem \ref{thm:2}.
Let us first present a sufficient condition for non-explosion of growth-fragmentations, which slightly generalizes the approach in \cite{Bertoin:growth}.

\subsection{A sufficient condition for non-explosion}\label{sec:gf}
A Feller process $X = (X(t), t\geq 0)$ is called a {\bf cell process}, if it has c\`adl\`ag path on $(0,\infty)\cup \{\partial\}$ with no positive jumps. We refer to $\partial$ as a cemetery point and denote the lifetime of $X$ by $\zeta:= \inf \left\{ t\geq 0~:~ X(t)=\partial \right\}\in [0,\infty]$. 
 For every $ x\geq 0$ we write $\muA_{x}$ for the law of $X$ with initial value $X(0) = x$ and $E_{x}$ for mathematical expectation under $\muA_{x}$.


As we have discussed in Section \ref{sec:HGF}, to study the growth-fragmentation associated with $X$, we first want an ordering of the jumps of $X$, which is necessary to rigorously build a cell system driven by $X$. Furthermore, we need a sufficient condition for the {\it non-explosion} of the cell system, that is for every $t\geq 0$ the multiset of the sizes of all cells alive at time $t$ is locally finite. For these purposes, we henceforth suppose the following hypothesis for $X$, which is reminiscent of that in Theorem 1 in \cite{Bertoin:growth}.


\begin{enumerate}[label={\bf [H]}, ref={\bf [H]},leftmargin=3.0em]
\item \label{H} There exists a measurable function $f \colon [0,\infty)\times ((0,\infty)\cup\{\partial\}) \to [0, \infty)$, with $\ft{r}{\partial}\equiv 0$ for every $r\geq 0$, which fulfills
 \begin{equation}\label{eq:exc}
    \inf_{ r< l, x>a} \ft{r}{x}>0, \quad \text{ for every } a, l>0,
  \end{equation}
such that for every $x>0$ and every $s,t\geq 0$, there is
\begin{equation}
  \Exp[x]{\ft{s+t}{X(t)} + \sum_{0\leq r\leq t} \ft{s+r}{-\Delta X(r)} } \leq \ft{s}{x}. 
\end{equation}
\end{enumerate}

\begin{eg}\label{eg:hom}
For $x>0$, let $P_x$ be the law of the homogeneous cell process $\XO=x \exp(\xi)$. Fix $q\geq 2$ and $K \geq \kappa(q)$, we have by an analogue of \eqref{eq:hom} that for every $x>0$ and every $s,t\geq 0$ 
\begin{linenomath}\begin{align} 
&\Exp[x]{X^{(0)}(t)^q \e^{-K (t+s)} + \sum_{0\leq r\leq t}|\Delta X^{(0)}(r)|^q \e^{-K (r+s)}}\\
&=  \left(\e^{(\Phi(q) -K) t} + \frac{\kappa(q) -\Phi(q)}{K -\Phi(q)} (1 - \e^{(\Phi(q)-K) t}) \right) 
x^q \e^{-K s} \leq x^q \e^{-K s}.\end{align}\end{linenomath}
So $\XO$ satisfies \ref{H} with the function $(t,x)\mapsto x^q \e^{-K t}$. 
\end{eg}
From now on we fix a function $f$ such that \ref{H} holds for $X$. 
In particular \ref{H} entails that for every $x>0$
$$\sum_{r\geq 0} \ft{r}{-\Delta X(r)}< \infty \quad P_{x}\text{-almost surely}.$$ 
Hence we may naturally enumerate the jump times of $X$ by listing them in a sequence $(t_i)_{ i\in \N}$ such that $(\ft{t_i}{-\Delta X(t_i)})_{ i\in \N}$ is decreasing, and thus reproduce the construction in Section \ref{sec:HGF} to build a cell system $\Xc:= (\Xc_u, ~ u\in \U)$ driven by $X$, starting from an ancestor of initial size $x>0$, with birth time $b_u$ and life length $\zeta_u:= \inf\{t\geq 0~:~\Xc_u(t)=\partial\}$. Denote the sizes of the cells alive at time $t\geq 0$ by the multiset 
$$ \Xs(t) := \multiset{\Xc_u(t- b_u) ~:~ u \in \U, b_u\leq t< b_u+ \zeta_u},$$
then $(\Xs(t),t\geq 0 )$ is a {\bf growth-fragmentation process driven by $X$}. We write $\Pc_{x}$ for the law of $\Xc$ and $\Ps_{x}$ for the law of $\Xs$ under $\Pc_{x}$. 
It is intuitively clear that the law of $\Xs$ is independent of the enumeration method. 

For every non-negative measurable function $h: (0,\infty) \to [0,\infty)$ and every multiset $\ms$ with elements in $(0,\infty)$, introduce the notation 
$$\prm{\ms}{h} := \sum_{y\in \ms} h(y) \in [0, \infty].$$
Let us define for every $s\geq 0$ a space 
$\MF^s$: a multiset $\ms \in \MF^s$, if $\ms$ has elements in $(0, \infty)$ and $\prm{\ms}{\ft{s}{\cdot}} < \infty$. 
\begin{lem}\label{lem:exc}
Suppose that $X$ satisfies \ref{H} with a function $f$. 
Then we have for every $x> 0$ that
\begin{equation}
  \Exps[x]{ \prm{\Xs(t)}{f(s+t, \cdot)} } \leq \ft{s}{x}, \quad \text{for all}~ t, s \geq 0,
\end{equation}
where $\mathbf E_{x}$ denotes the mathematical expectation under $\Ps_{x}$. So we have $\Ps_{x}$-almost surely $\Xs(t) \in \MF^t$.  
\end{lem} 
Lemma \ref{lem:exc} encompasses Theorem 1 in \cite{Bertoin:growth} for the case when $f$ only depends on the $x$ variable, i.e. $\ft{t}{x} \equiv f(x)$ for every $x,t\geq 0$. In that case $f$ is a so-called {\it excessive function} for $\Xs$. In the same spirit, we may refer to $f$ as a {\it time-dependent excessive function} for $\Xs$. 
\begin{proof}
The proof is an adaptation of arguments of Theorem 1 in \cite{Bertoin:growth}. 
We may assume that $\Xs$ is associated with a cell system $\Xc$ of law $\Pc_{x}$ and write $\mathcal E_{x}$ for mathematical expectation under $\Pc_{x}$. 
We will prove that the sequence
$$\Sigma(i) := \sum_{|u|\leq i, b_u \leq t} \ft{s+t}{\Xc_u (t- b_u)} + \sum_{|v|=i, b_v\leq t}\sum_{b_v \leq r\leq t} \ft{s+r}{-\Delta \Xc_v ( r-b_v)}, \quad i\in \N$$
is a non-negative super-martingale, then $\Sigma(\infty) = \lim_{i\to \infty} \Sigma(i)$ exists almost surely and $\Sigma(\infty)\geq \prm{\Xs(t)}{f(s+t, \cdot)}$. We thus deduce from Fatou's lemma that 
\begin{equation}
   \Exps[x]{\prm{\Xs(t)}{f(s+t, \cdot)}}  \leq  \Expc[x]{\Sigma(0)}
  = \Exp[x]{\ft{s+t}{X(0)} + \sum_{0\leq r\leq t} \ft{s+r}{-\Delta X(r)} } \leq  \ft{s}{x} ,
\end{equation}
where the last inequality derives from \ref{H}. 

So it remains to prove that $\Sigma(i)$ is a super-martingale.  
For every $v$ with $|v|= i$, given $\F_{i-1}:= \sigma(\Xc_u, |u|\leq i-1)$ we have by \ref{H} that 
$$\Expcondc[x]{\ft{s+t}{\Xc_v (t- b_v)} + \sum_{b_v \leq r\leq t} \ft{s+r}{-\Delta \Xc_v (r-b_v)} }{\F_{i-1}} \leq  \ft{s+b_v}{\Xc_v(0)} .$$
Summing over $v$ of $i$-th generation, we get that 
\begin{linenomath}\begin{align} 
  &  \Expcondc[x]{\sum_{|v|= i, b_v \leq t} \ft{s+t}{\Xc_v (t-b_v)} +  \sum_{|v|=i, b_v\leq t}\sum_{b_v \leq r\leq t} \ft{s+t}{-\Delta \Xc_v (r-b_v)}  }{\F_{i-1}}  \\
& \leq  \sum_{|v|= i,b_v \leq t} \ft{s+b_v}{\Xc_v(0)}  = \sum_{|u|= i-1, b_u \leq t} \sum_{b_u \leq r \leq t} \ft{s+r}{-\Delta \Xc_u(r-b_u)}.
  \end{align}\end{linenomath}
Adding $\sum_{|u|\leq i-1, b_u \leq t} \ft{s+t}{\Xc_u (t-b_u)}$ to both sides of inequality, we conclude that 
$$\Expcondc[x]{\Sigma(i)}{\F_{i-1}}\leq \Sigma(i-1), $$
which means that $\Sigma(i)$ is a super-martingale. 
\end{proof}

Let $\SM$ be the class of all multisets $\ms$ on $(0, \infty)$, which has only finitely many elements in $[a, \infty)$ for every $a>0$. Note that each $\ms \in \SM$ corresponds to a Radon measure and \eqref{eq:exc} ensures that $\MF^s \subset \SM$ for every $s\geq 0$. On account of Lemma \ref{lem:exc}, we can hence view the growth-fragmentation $\Xs$ as a stochastic process with values in $\SM$, which means that $\Xs$ does not explode. The space $\SM$ is endowed with the following topology: 
\begin{defn}\label{defn:SM} 
We denote the cardinality of a multiset $\mathcal J$ by $|\mathcal J|$. 
A sequence $(\ms_n)_{n\in \N} \in \SM$ converges to $\ms \in \SM$ if and only if
for all $r\in (0,\infty)$ such that $\ms\cap \{r\} =\emptyset$ there is $ |\ms_n\cap [r,\infty)| \to |\ms\cap [r,\infty)|$.
\end{defn}
The advantage of endowing $\SM$ with this topology is that it is a Polish space (homeomorphic to a complete and separable metric space), see Theorem 2.1 and Theorem 2.2 in \cite{Vague}. It is known from Lemma 2.1 in \cite{Vague} that convergence in $\SM$ implies vague convergence. See \cite{Vague} for more properties of $\SM$.

We next introduce a {\it truncate} operation on $\Xc$ tailored for our future purpose, which is different from the one  in the proof of Proposition \ref{prop:1}. For every $\epsilon> 0$, we obtain a truncated system $\cutc{\Xc} = (\cutc{\Xc}_u, u\in \U)$, by killing each cell process at the first time $s\geq 0$ when its size is less than or equal to $\epsilon$, together with its future (born at time $>s$) descendants. 
Specifically, let us denote for every $u\in \U$ its ancestral lineage by $\A_u:=( \A_u(t), t\geq 0)$, i.e.
$$ \A_u(t) := \sum_{n\leq |u|-1} \Xc_{{[u]}_n}(t-b_{[u]_n}) \ind{b_{[u]_n}\leq t< b_{[u]_{n+1}}} + \Xc_u(t- b_{u}) \ind{b_{u}\leq t}, \quad t\geq 0,$$ 
where ${[u]}_n$ denotes $u$'s ancestor at the $n$-th generation for all $n\leq |u|$, then we have that
   \begin{equation}\label{eq:cut}
     \cutc{\Xc}_u (t):=\begin{cases}
       \Xc_u(t), & \text{if } \inf_{0\leq r\leq t+b_u} \A_u(r)>\epsilon. \\
       \partial, & \text{otherwise}.
     \end{cases}
   \end{equation} 
 Let $\cutc{\Xs}$ be the point process on $(0,\infty)$ associated with $\cutc{\Xc}$: 
$$\cutc{\Xs}(t) = \multiset{\cutc{\Xc}_u (t-b_u): ~ u\in \U, b_u\leq t, \cutc{\Xc}_u (t-b_u)\neq \partial }, \quad t\geq 0.$$
\begin{lem}\label{lem:CV}
Suppose that $X$ satisfies \ref{H}. Then for every $x\geq 0$ and every $t\geq 0$, under the topology of $\SM$ the multiset $\cutc{\Xs}(t)$ converges $\Ps_{x}$-almost surely to $\Xs(t)$ as $\epsilon \downarrow 0+$.
\end{lem}
\begin{proof}
We first note that if a cell $u \in \U$ is alive at time $t\geq 0$ with $\Xc_u(t-b_u)>0$, then $\Ps_{x}$-almost surely its ancestral lineage has a size bounded away from $0$ before time $t$, i.e. $\inf_{0\leq r\leq t} A_u(r)>0$. So there exists $\epsilon>0$ small enough such that $\cutc{\Xc}_u(r-b_u) = \Xc_u(r-b_u)$ for all $b_u \leq r\leq t$, and we have $\Ps_{x}$-almost surely 
$$\lim_{\epsilon \to 0+}\cutc{\Xc}_u(t-b_u)\ind{t\geq b_u} = \Xc_u(t-b_u)\ind{t\geq b_u}. $$
We hence obtain by the monotone convergence that for every $a>0$, $\Ps_{x}$-almost surely 
\begin{align} 
    \lim_{\epsilon \to 0}|\cutc{\Xs}(t)\cap [a,\infty)| &= \lim_{\epsilon \to 0} \sum_{u\in \U}\ind{\cutc{\Xc}_u(t -b_u)\geq a}\cutc{\Xc}_u(t -b_u)\ind{t\geq b_u}\\
&=  \sum_{u\in \U}\lim_{\epsilon \to 0} \ind{\cutc{\Xc}_u(t -b_u)\geq a}\cutc{\Xc}_u(t -b_u)\ind{t\geq b_u} = |\Xs(t) \cap [a,\infty)|,
  \end{align}
which means that $\cutc{\Xs}(t)$ converges $\Ps_{x}$-almost surely to $\Xs(t)$ in $\SM$.
\end{proof}

We observe that the truncated system $\cutc{\Xc}$ has discrete temporal branching structure, since for each c\`adl\`ag process the set of jump times with sizes of jumps $<-\epsilon$ is discrete.
By the same arguments as the proof of Proposition 2 in \cite{Bertoin:growth}, we deduce from this observation and Lemma \ref{lem:CV} that $\Xs$ has the temporal branching property. 
To describe this property, let us define a family $(\rho_{s,t}, t\geq s \geq 0)$, where each $\rho_{s,t}$ is a probability kernel from $\MF^s$ to $\MF^t$, in the following way. 
Given a multiset $\mathcal J\in \MF^s$, we may construct a family of independent random multisets $(\ms_y, y\in \mathcal J)$, such that each $\ms_y$ has the law of $\Xs(t-s)$ under $\Pc_{y}$. Then $\mathcal J^t:= \biguplus_{y\in \mathcal J} \ms_y\in \MF^t$, since it follows from Lemma \ref{lem:exc} that
$$ \prm{\mathcal J^t}{\ft{t}{\cdot}} = \sum_{y\in \mathcal J} \Exps[y]{\prm{\Xs(t-s)}{\ft{t}{\cdot} } } 
\leq \sum_{y\in \mathcal J}\ft{s}{y} = \prm{\mathcal J}{ \ft{s}{\cdot}} < \infty.$$
We hence define $\rho_{s,t}(\mathcal J,\cdot)$ by the law of $\mathcal J^t$. 

\begin{prop}[Temporal branching property]\label{prop:TBP}
  Suppose that $X$ possesses a function $f$ that satisfies \ref{H}, then for every $t \geq  s\geq 0$ and every $x>0$, the conditional distribution of $\Xs(t)$ under $\Ps_{x}$ given $( \Xs(r), 0\leq r\leq s)$ is $\rho_{s,t}(\Xs(s), \cdot)$.
\end{prop}

\begin{rem}
  One may easily extend the analysis in this section to time-inhomogeneous Markov processes. 
Let $X$ is a time-inhomogeneous cell process and write $\muA_{s,x}$ for the law of $X$ starting at time $s\geq 0$ with initial size $x\geq 0$. Then the counterpart of condition \ref{H} is that there exists a function $f$ that satisfies \eqref{eq:exc} and for every $x>0$ and every $s\geq 0$, 
\begin{equation}
  \Exp[s,x]{\ft{t}{X(t)} + \sum_{s\leq r\leq t} \ft{r}{-\Delta X(r)} } \leq \ft{s}{x}, \quad\text{for all } t\geq s,
\end{equation}
where $E_{s,x}$ means mathematical expectation under $\muA_{s,x}$. 
Under this condition, one may easily build a cell system driven by $X$ (with the life path of each $\Xc_u$ scaled by the universal time) and check that the system does not explode by an analogue of Lemma \ref{lem:exc}. Details shall be left to interested readers. 
\end{rem}


\subsection{Bifurcators}\label{sec:bifurcator}
For every $x>0$, let $P_x$ and $Q_x$ be respectively the laws of two cell processes $X$ and $Y$, both starting from $x$. 
We now give a formal definition of bifurcators of cell processes, which extends both Definition 2 by Pitman and Winkel \cite{PitmanWinkel} and the present Definition \ref{defn:bifH} for homogeneous cell processes. 
\begin{defn}\label{defn:bif}
A bivariate process $(X',Y')$ is called a {\bf bifurcator} of branches $X$ and $Y$, if it satisfies the following properties:
  \begin{enumerate}[label=(\roman*)]
 \item  For every $x> 0$, write $\Pb_{x}$ for the joint distribution of $(X',Y')$ with $X'(0)=Y'(0)=x$.  Under $\Pb_{x}$, each component $X'$ and $\tildeX'$ has the law $P_{x}$ and $\tildeP_{x}$ respectively, that is, the two marginal distributions of $\Pb_{x}$ are $P_{x}$ and $\tildeP_{x}$. 
\item Let $\tb:= \inf \{t\geq 0: X'(t) \neq \tildeX' (t) \}$. For every $x> 0$, conditionally on $\{\tb<\infty\}$, there is 
  \begin{equation}
  X'(\tb)+ \tildeX'(\tb) = X'(\tb-)= \tildeX'(\tb-), \quad \Pb_{x}-a.s.\label{eq:bif}
  \end{equation}
\item  (Asymmetric Markov branching property)
For every $x> 0$, the process $(X'(t), \tildeX'(t), \ind{\tb>t})_{t\geq 0}$ under $\Pb_{x}$ is Markovian. Specifically, conditionally given $\tb>t$, the process $(X'(r), \tildeX'(r))_{r\geq t}$ has distribution $\Pb_{X'(t)}$; conditionally given $\tb\leq t$, $(X'(r), \tildeX'(r))_{r\geq t} $ is a pair of independent processes of respective laws $P_{ X'(t)}$ and $\tildeP_{\tildeX'(t)}$. 
\end{enumerate}
If such a bifurcator $(X',Y')$ exists, then we say {\bf $X$ and $\tildeX$ can be coupled to form a bifurcator}.
\end{defn}

\begin{rem} 
We know from \eqref{eq:bif} that if $\tb < \infty$, then \eqref{eq:bif} implies that $\tb$ is a jump time of both $X'$ and $Y'$, which is almost surely strictly positive and strictly smaller than the lifetimes of $X'$ and $Y'$. Define a filtration $(\Fb_t)_{ t\geq 0}$ by the usual augmentation of $\sigma(X'(r), \tildeX' (r), ~0\leq r\leq t)$, note that $\tb$ is a $(\Fb_t)$-stopping time and each component $X'$ or $\tildeX'$ satisfies the strong Markov property. 
\end{rem}


We next state a sufficient condition for growth-fragmentations based on different cell processes to have the same distribution, which is the main purpose of this work. Suppose that \ref{H} holds for both $X$ and $Y$, then we know from the precedent subsection that we can construct two non-exploded growth-fragmentations $\Xs$ and $\Ys$ associated with $X$ and $Y$ respectively.  
Note that \ref{H} entails that for every $x>0$ and every $s\geq 0$,  
$$   \Exp[x]{ \sum_{r\geq 0 } f(s+r, -\Delta X(r)) } \leq f(s,x).$$
However, we shall need the stronger assumption:
  \begin{enumerate}[label={\bf [H$\mathbf{\eta}$]}, ref={\bf [H$\mathbf{\eta}$]},leftmargin=3.0em]
\item \label{Heta} 
there exist a function $\gt$ that satisfies \eqref{eq:exc} and a constant $\eta <1$, such that for every $x>0$ and every $s\geq 0$,  
$$   \Exp[x]{ \sum_{r\geq 0 } \gt(s+r, -\Delta X(r)) } \leq \eta \gt(s,x).$$
\end{enumerate}

\begin{thm}\label{thm:1}
Let $X$ and $\tildeX$ be two cell processes that both satisfy \ref{H} and \ref{Heta}. Suppose that $X$ and $\tildeX$ can be coupled to form a bifurcator, then for every $x>0$, two Markovian growth-fragmentations $\Xs$ and $\Ys$ driven respectively by $X$ and $Y$, both starting from $x$, have the same finite-dimensional distribution.  
\end{thm}

\subsection{Proof of Theorem \ref{thm:1}}
Let us briefly explain the idea of the proof. 
Fix $x>0$, let $\Xc$ and $\Yc$ be cell systems associated with $X$ and $Y$ respectively, with respective laws $\Pc_x$ and $\tildePc_x$. 
For every $\epsilon>0$, let $\cutc{\Xs}$ be the process associated with the truncated cell system $\cutc{\Xc}$ derived from $\Xc$ as in \eqref{eq:cut}, by killing each cell together with its future descent when its size becomes less than or equal to $\epsilon$. Similarly we define $\cutc{\Ys}$. 
We shall prove for every $\epsilon>0$ that $\cutc{\Xs}$ under $\Pc_x$ has the same law as $\cutc{\Ys}$ under $\tildePc_x$. Then letting $\epsilon \to 0+$, we conclude from Lemma \ref{lem:CV} that $\Xs$ and $\Ys$ have the same finite-dimensional distribution.

Let us fix an arbitrary $\epsilon >0$. To prepare for the proof that $\cutc{\Xs}$ and $\cutc{\Ys}$ have the same law, we construct a family of bivariate processes $((X_v,Y_v),v\in \Utwo)$ (recall that $\Utwo= \bigcup_{n\in \N}\left\{ \vl,\vr \right\}^n$ is the binary tree) in the following way. Since $X$ an $Y$ can be coupled to form a bifurcator, there exists a bifurcator with distribution $(\Pb_y, y>0)$, whose marginal distributions are $P_y$ and $Q_y$ under $\Pb_y$. Then we let $(X_{\emptyset}, Y_{\emptyset})$ be a bifurcator with law $\Pb_{x}$ and write $\beta_{\emptyset}:= 0$ for the birth time of $\emptyset$.  
Suppose by induction that we have built for a certain $v\in \Utwo$ a bifurcator $( X_v, Y_v)$ with birth time $\beta_v$. Write $\tb_v:= \inf \{ t\geq 0: X_{v}(t) \neq \tildeX_{v} (t) \}$ for the switching time of this bifurcator, $T^{X}_v:= \inf\left\{ t\geq 0 :~ X_v(t) \leq \epsilon\right\}$ for the first time when $X_v$ is smaller than $\epsilon$, and $\tilde{T}^{X}_v:= \inf\left\{ t\geq 0 :~ -\Delta X_v(t) > \epsilon\right\}$ for the first time when $X_v$ has a jump of size greater than $\epsilon$, then we define the lifetime of $v$ by 
$$\dtb_v:= \tb_v\wedge T^{X}_v\wedge \tilde{T}^{X}_v,$$
then $\dtb_v$ is a $(\Fb^{v}_t)$-stopping time, where $\Fb^{v}_t$ is the augmentation of $\sigma((X_{v}(r), \tildeX_{v} (r)), ~0\leq r\leq t)$. 
At the lifetime $\dtb_v$, we distinguish the following two situations.
\begin{itemize}
 \item If $\dtb_v =T^{X}_v< \tilde{T}^{X}_v\wedge \tb_v$ or $\dtb_v = \infty$, then $v$ is {\bf killed} at its lifetime $\dtb_v$. Further, we agree that for every $w\in \Utwo\setminus \{\emptyset\}$, $vw$ is also killed, with $\beta_{vw}=\infty$, $\dtb_{vw} = 0$ and $X_{vw} \equiv Y_{vw}\equiv \partial$. 
As $\tb_v\wedge \tilde{T}^{X}_v$ is almost surely strictly positive, this situation also covers the case when $\dtb_v=0$ (if and  only if $T^{X}_v=0$, i.e. $X_v(0) \leq \epsilon$).
\item Otherwise, $v$ {\bf branches} at its lifetime $\dtb_v$, giving birth to two independent bifurcators $(X_{v\vl},Y_{v\vl})$ and $(X_{v\vr},Y_{v\vr})$ with respective distributions $\Pb_{a_{v\vl}}$ and $\Pb_{a_{v\vr}}$, where 
  \begin{equation}\label{eq:XY0}
    (a_{v\vl},a_{v\vr} ):= (X_v(\dtb_v), -\Delta X_v(\dtb_v)).
  \end{equation}
 Set their birth time by $\beta_{v\vl} =\beta_{v\vr}:= \beta_v +\dtb_v$.
We further {\bf mark} $v$ if  $\dtb_v =\tb_v\leq T^{X}_v \wedge \tilde{T}^{X}_v$ (we also say that we mark the branching event at the death of $v$), so $v$ is {\bf non-marked} if $\dtb_v =\tilde{T}^{X}_v< \tb_v$. 
Using the junction relation \eqref{eq:bif} of the bifurcator, we also have
\begin{equation}
 (a_{v\vl},a_{v\vr} )=\begin{cases}
    (-\Delta Y_v(\dtb_v), Y_v(\dtb_v)), & \text{if } v \text{ is marked}, \\
     (Y_v(\dtb_v), -\Delta Y_v(\dtb_v)), & \text{if } v \text{ is non-marked}.
  \end{cases}\label{eq:mark}
\end{equation} 
Note that if $v$ is non-marked, then $a_{v\vr}> \epsilon$ always holds; but if $v$ is marked, then it is possible that $a_{v\vr}\leq \epsilon$, which means that $v\vr$ is immediately killed and $\dtb_{v\vr} = 0$. 
In both marked and non-marked cases, it is possible that $a_{v\vl}\leq \epsilon$ and $v\vl$ is immediately killed with $\dtb_{v\vl} = 0$.
\end{itemize}

We continue so on and so forth to construct all generations of the family $((X_v,Y_v),v\in \Utwo)$ and finally define a process
$$\Ws_{(X,Y)}(t):= \multiset{X_v(t-\beta_v)~:~ v\in \Utwo, \beta_v \leq t < \beta_v+ \dtb_v}, \qquad t\geq 0.$$
Note by construction that every element of $\Ws_{(X,Y)}(t)$ is larger than $\epsilon$.

A notable feature of this system is that, roughly speaking, $\Ws_{(X,Y)}$ is symmetric, i.e. its law is invariant under the permutation of labels $X$ and $Y$.

\begin{lem}\label{lem:WW}
  $\Ws_{(X,Y)}$ has the same law as $\Ws_{(Y,X)}$. 
\end{lem}
\begin{proof}
Given the family $((X_v,Y_v),v\in \Utwo)$ constructed as above, let us define recursively a bijection $h:\Utwo \to \Utwo$ with $h(\emptyset):= \emptyset$, such that for every $v\in \Utwo$ we have $(h(v\vl), h(v\vr)) := (h(v)\vr,h(v)\vl)$ if $h(v)$ is marked, and $(h(v\vl), h(v\vr)) := (h(v)\vl,h(v)\vr)$ if $v$ is non-marked or $v$ is killed.
We next describe the dynamics of $((Y'_{v}, X'_{v}), v\in \Utwo):= ((Y_{h(v)}, X_{h(v)}), v\in \Utwo)$ as a bivariate system generated by the bifurcator $(Y,X)$. Specifically, define $T^{Y}_v$ and $\tilde{T}^{Y}_v$ for $Y_v$ in the same way as $T^{X}_v$ and $\tilde{T}^{X}_v$, then the lifetime of each $(Y'_{v}, X'_{v})$ is $\dtb'_{v}:= \tb_{h(v)}\wedge T^{Y}_{h(v)}\wedge \tilde{T}^{Y}_{h(v)}$, which is equal to $\dtb_{h(v)}$. Indeed, since $X_{h(v)}(t) = Y_{h(v)}(t)$ for all $t< \tb_{h(v)}$, we find that
\begin{itemize}
 \item If $\dtb_{h(v)} =T^{X}_{h(v)}< \tilde{T}^{X}_{h(v)}\wedge \tb_{h(v)}$ or $\dtb_{h(v)} = \infty$, then $T^{Y}_{h(v)}= T^{X}_{h(v)}$ and $T^{Y}_{h(v)}< \tilde{T}^{Y}_{h(v)}\wedge \tb_{h(v)}$;
\item If $\dtb_{h(v)} =\tb_{h(v)}\leq T^{X}_{h(v)} \wedge \tilde{T}^{X}_{h(v)}$, then $\tb_{h(v)}\leq T^{Y}_{h(v)} \wedge \tilde{T}^{Y}_{h(v)}$;
\item if $\dtb_{h(v)} =\tilde{T}^{X}_{h(v)}< \tb_{h(v)}$, then $\tilde{T}^{Y}_{h(v)}=\tilde{T}^{X}_{h(v)}< \tb_{h(v)}$ and $\tilde{T}^{Y}_{h(v)}\leq T^{Y}_{h(v)}$. 
\end{itemize}
At the lifetime $\dtb_{h(v)}$, $v$ is killed in the first case; in the other two cases, $v$ generates two independent bifurcators $(Y'_{v\vl}, X'_{v\vl})$ and $(Y'_{v\vr}, X'_{v\vr})$ of respective laws $\Pb_{a_{h(v\vl)}}$ and $\Pb_{a_{h(v\vr)}}$.
 It follows from \eqref{eq:mark} and the construction of $h$ that for every $v\in \Utwo$
$$(a_{h(v\vl)},a_{h(v\vr)} ) = (Y_{h(v)}(\dtb_{h(v)}), -\Delta Y_{h(v)}(\dtb_{h(v)})) .$$  
We hence conclude that the process
$$\Ws'(t):=\multiset{Y'_{v}(t-\beta'_{v}):~ v\in \Utwo, \beta'_{v} \leq t < \beta'_{v}+ \dtb'_{v}}= \multiset{Y_{h(v)}(t-\beta_{h(v)}):~ v\in \Utwo, \beta_{h(v)} \leq t < \beta_{h(v)}+ \dtb_{h(v)}},$$
is a copy of $\Ws_{(Y,X)}$. On the other hand, since $h$ is a bijection and recall that for every $v\in \Utwo$, $X_v(t) = Y_v(t)$ for all $t< \dtb_v$, then clearly $\Ws'= \Ws_{(X,Y)}$.
\end{proof}

We next consider the process associated with the left-most branch $\Rm=\{\emptyset, \vl, \vl\vl, \ldots\}$, that is
\begin{equation}\label{eq:A}
  \A_{\Rm}(t):= \sum_{n\geq 0} \ind{\beta_{\vl^n}\leq  t< \beta_{\vl^n}+ \dtb_{\vl^n}} X_{\vl^n}(t - \beta_{\vl^n})= \sum_{n\geq 0} \ind{\beta_{\vl^n}\leq  t< \beta_{\vl^n}+ \dtb_{\vl^n}} Y_{\vl^n}(t - \beta_{\vl^n}) ,\quad t\geq 0. 
\end{equation}
where $\vl^n:= (\vl\vl\ldots \vl)\in \{\vl,\vr\}^n$ and $\vl^0:=\emptyset$. 
Let $N:=\inf \left\{ n\in \N:~ \vl^{n}\text{ is killed}\right\}$, with convention $\inf\{\emptyset\} = \infty$. If $N< \infty$, then $\A_{\Rm}(t) = \partial$ for all $t\geq \beta_{\vl^N}+ \dtb_{\vl^N}$. By concatenating $\A_{\Rm}$ with the segment of $X_{\vl^N}$ after its lifetime $\dtb_{\vl^N}$, we define
\begin{equation}\label{eq:AX}
   \A^X_{\Rm}(t):=\begin{cases}
    \A_{\Rm}(t) , & t< \beta_{\vl^{N}}+ \dtb_{\vl^N}. \\
    X_{\vl^N}(t - \beta_{\vl^N}),, & t\geq \beta_{\vl^{N}}+ \dtb_{\vl^N}.
  \end{cases}
\end{equation}
We agree that $\A^X_{\Rm} = \A_{\Rm}$ if $N= \infty$. 
 \begin{lem}\label{lem:Heta}
Suppose that \ref{Heta} holds for both $X$ and $Y$. Then the process $\A^X_{\Rm}$ has the law of $P_x$ (the law of $X$ starting from $x$), and the process derived from $\A^X_{\Rm}$ by killing at $\zeta^X_{\Rm}:= \inf\left\{t\geq 0:~ \A^X_{\Rm}(t)\leq \epsilon \right\}$ is $\A_{\Rm}$. 
\end{lem}

\begin{proof}
 It should be intuitive that $\A^X_{\Rm}$ has the law of $P_x$ because of the construction \eqref{eq:XY0} and the strong Markov property of $X$; however, it is a priori possible that none of $\Rm$ is killed and their birth times accumulate to a finite limit, i.e. $N=\infty$ and $\lim_{n\to \infty}\beta_{\vl^n} < \infty$, then $\A^X_{\Rm}$ is killed at this limit time, thus does not have the law of $P_x$. We shall prove that this case does not happen, thanks to the assumption \ref{Heta}. Therefore, almost surely there are only two possible situations: either $N<\infty$, or $N=\infty$ \& $\lim_{n\to \infty} \beta_{\vl^n} =\infty$, so we deduce from the strong Markov property of $X$ that $\A^X_{\Rm}$ indeed has the law of $P_x$. Further, we easily check that $\zeta^X_{\Rm}= \beta_{\vl^{N}}+ \dtb_{\vl^N}$ when $N<\infty$ and $\zeta^X_{\Rm}= \infty$ when $N= \infty$, then the second part of the claim follows.   

So it remains to prove that if $N=\infty$, which means that none of $(\vl^n)_{n\in \N}$ is killed, then $\lim_{n\to \infty}\beta_{\vl^n} = \infty$. We consider separately the following two situations. 

In the first situation there are infinitely many marked elements in $\Rm$, and we list all of them in a sequence $(\vl^{n_i})_{i\in \N} \subset \Rm$ with $n_i\uparrow \infty$. 
Let $\G_n:= \sigma(X_{\vl^{j}}, Y_{\vl^{j}}, j\leq n)$ and $\gt_Y$ be a function such that \ref{Heta} holds for $Y$ with $\eta_Y<1$, then 
$$M_i := \eta_Y^{-i} \sum_{r \geq 0} \gt_Y(\beta_{\vl^{n_i}} + r, -\Delta Y_{\vl^{n_i}}(r)),\quad i\in \N $$ 
is a non-negative $\G_{n_i}$-super-martingale. 
Indeed, consider the ancestral lineage of $\vl^{n_{i+1}}$ for the $Y$-side, shifted to the left by $\beta_{\vl^{n_i}\vl}$ ($\vl^{n_i}\vl$ means $\vl^{n_i+1}$), that is 
$$A^Y_{i+1}(t):= \sum_{n_i+1 \leq k < n_{i+1}} \ind{\beta_{\vl^k}\leq  t + \beta_{\vl^{n_i}\vl}< \beta_{\vl^{k}}+ \dtb_{\vl^{k}}} Y_{\vl^k}(t+ \beta_{\vl^{n_i}\vl}-\beta_{\vl^k} ) + \ind{  t + \beta_{\vl^{n_i}\vl}\geq  \beta_{\vl^{n_{i+1}}}}Y_{\vl^{n_{i+1}}}(t + \beta_{\vl^{n_i}\vl}- \beta_{\vl^{n_{i+1}}}) ,\quad t\geq 0 ,$$
with $A^Y_{i+1}(0)= Y_{\vl^{n_i}\vl}(0)$. Then  
$$M_{i+1} \leq \eta_{Y}^{-(i+1)}\sum_{r \geq 0} \gt_{Y}(\beta_{\vl^{n_i}\vl} + r, -\Delta A^Y_{i+1}(r)).$$
Observing that these segments are connected by only non-marked branching events and using \eqref{eq:mark}, we hence deduce by the strong Markov property of $Y$ that conditionally on $\G_{n_i}$, $A^Y_{i}$ has distribution $Q_{y}$ with $y:= Y_{\vl^{n_i}\vl}(0)$. As $Y$ satisfies \ref{Heta}, we have
\begin{linenomath}\begin{align} 
  & \Expcond{M_{i+1}}{\G_{n_i}} \leq \Expcond{\eta_{Y}^{-(i+1)}\sum_{r \geq 0} \gt_{Y}(\beta_{\vl^{n_i}\vl} + r, -\Delta A^Y_{i+1}(r))}{\G_{n_i}} \\
&\leq   \eta_{Y}^{-i}\gt_{Y}(\beta_{\vl^{n_i}\vl}, Y_{\vl^{n_i}\vl}(0)) = \eta_{Y}^{-i}\gt_{Y}(\beta_{\vl^{n_i}}+\dtb_{\vl^{n_i}}, -\Delta Y_{\vl^{n_i}}(\dtb_{\vl^{n_i}}))  \leq M_{i},
  \end{align}\end{linenomath}
where the equality follows from \eqref{eq:mark} as $\vl^{n_i}$ is marked. We conclude that $M_i$ is a non-negative super-martingale and hence $M_i$ converges almost surely to a limit as $i\to \infty$.
Multiplying the last display by $\eta_Y^{n_i}$, we have 
$$ \gt_{Y}(\beta_{\vl^{n_i}\vl}, X_{\vl^{n_i}\vl}(0)) = \gt_{Y}(\beta_{\vl^{n_i}} + \dtb_{\vl^{n_i}}, -\Delta Y_{\vl^{n_i}}(\dtb_{\vl^{n_i}})) \leq \eta_{Y}^{i}M_i \to 0 \quad \text{almost surely}.$$ 
As $\gt_Y$ satisfies \eqref{eq:exc}, it follows that in the event $\lim_{n\to \infty} \beta_{\vl^n}  <\infty$, there is 
$\lim_{i\to \infty}  X_{\vl^{n_i}\vl}(0)= 0$. This is absurd as we have assumed that no element in $\Rm$ is killed. 
 
In the second situation, there are infinitely many non-marked branching elements in $\Rm$. 
Consider for each $k\in \N$ the ancestral lineage of $\vl^k$ for the side of $X$, i.e.
\begin{equation}
  \A^X_{\vl^k}(t):= \sum_{n=0}^{k-1} \ind{\beta_{\vl^n}\leq  t< \beta_{\vl^{n}\vl}} X_{\vl^n}(t-\beta_{\vl^n})+ \ind{\beta_{\vl^k}\leq  t} X_{\vl^k}(t- \beta_{\vl^k}),\quad t\geq 0 .
\end{equation}
Then for each $k\in \N$, we deduce from the strong Markov property of $X$ that $\A^X_{\vl^k}$ has law $P_x$. Let $\gt_X$ be a function such that \ref{Heta} holds for $X$ with constant $\eta_X<1$, then
\begin{equation}\label{eq:gtx}   \Exp{ \sum_{r\geq 0 } \gt_X(r, -\Delta \A^X_{\vl^k}(r)) } \leq \gt_X(0,x).\end{equation}
Suppose, by contradiction, that there exists a certain $M>0$ such that 
with probability $p_M>0$ there is $\lim_{n\to \infty} \beta_{\vl^n} < M$ and write $\inf_{t<M~\&~y\geq \epsilon}\gt_X(t,y)=:c_{M,\epsilon}>0$ as \eqref{eq:exc} holds for $\gt_X$. 
For every $k\in \N$, write $m_k$ for the number of non-marked particles in the set $\{\vl^i, i< k\}$, then we get that  
$$   \Exp{ \sum_{r\geq 0 } \gt_X(r, -\Delta \A^X_{\vl^k}(r)) } 
\geq  \Exp{\sum_{1\leq i\leq k-1}\ind{ \vl^i\text{ is non-marked}} \gt_X(\beta_{\vl^{i}}, -\Delta X_{\vl^{i}}(\dtb_{\vl^{i}})) } 
\geq p_M m_k c_{M,\epsilon}, $$
where the last inequality is obtained by restricting to the event $\lim_{n\to \infty} \beta_{\vl^n} < M$ and observing that $-\Delta X_{\vl^{i}}(\dtb_{\vl^{i}}) \geq \epsilon$ whenever $\vl^{i}$ is non-marked. 
Letting $k\to \infty$, we find a contradiction against \eqref{eq:gtx} since $m_k \to \infty$. We therefore prove the claim. 
\end{proof}

\begin{rem}Given the system $((X_{v},Y_{v}) ,v\in \Utwo)$, let us define a branch $\bar{v}:=(v_n\in \Utwo )_{n\geq 0}$ with $v_0 := \emptyset$ by setting recursively $v_{n+1}= v_n\vl$ if $v_n$ is non-marked and $v_{n+1}= v_n\vr$ if $v_n$ is marked. Then the branch $\A^Y_{\bar{v}}$ associated with the system $((X_{v},Y_{v}) ,v\in \Utwo)$ has the same law as $Y$. Indeed, recall that the system $((Y_{h(v)},X_{h(v)}) ,v\in \Utwo)$ defined as in the proof of Lemma~\ref{lem:WW} can be viewed as a system generated by the bifurcator $(Y,X)$, then applying Lemma~\ref{lem:Heta} to this system, we have that $\A^Y_{h(\Rm)}$ has the same law as $Y$. We observe that $\bar{v} = h(\Rm)$ by the construction of $h$, which entails our claim.
\end{rem}

\begin{lem}\label{lem:WX}
$\Ws_{(X,Y)}$ has the same law as $\cutc{\Xs}$.
\end{lem}
\begin{proof}
We first give some remarks on the process $\A^X_{\Rm}$ defined as in \eqref{eq:AX}. We know from Lemma~\ref{lem:Heta} that $\A^X_{\Rm}$ has law $P_x$, and $\A^X_{\Rm}$ killed at $\zeta_{\Rm}:=\inf\left\{t\geq 0:~ \A^X_{\Rm}(t)\leq \epsilon \right\}$ is $\A_{\Rm}$. 
For every $\vl^n\in \Rm$ such that $t:= \beta_{\vl^n}+\dtb_{\vl^n}\leq \zeta_{\Rm}$, we have by \eqref{eq:mark} that the size of the jump at $t$ is
$$y:= -\Delta\A^X_{\Rm}(t)= -\Delta X_{\vl^n}(\dtb_{\vl^n}) = X_{\vl^n\vr}(0).$$ 
Note that it is possible that $y>\epsilon$ or $y\leq \epsilon$: if $y\leq \epsilon$, then we know that the particle $\vl^n\vr$ (together with its progeny) is killed immediately, that is $\dtb_{\vl^n\vr}=0$.
On the other hand, for all $m\in \N$ and $t'\in (\beta_{\vl^m}, \beta_{\vl^m}+\dtb_{\vl^m})$, we have$ -\Delta\A^X_{\Rm}(t')\leq \epsilon.$

Let us now describe the dynamics of $\Ws_{(X,Y)}$ as the following truncated cell system. The cell system starts with a cell whose size evolves according to $\Xc_{\emptyset}:=A^X_{\Rm}$ with law $P_x$. By killing $\Xc_{\emptyset}$ at the time when entering $(0,\epsilon]$, we get $\cutc{\Xc}_{\emptyset}=\A_{\Rm}$. We next build the first generation. The daughter cells in the first generation born at time $>\zeta_{\Rm}$ are all killed. For each time $t\leq \zeta_{\Rm}$ with $y := -\Delta \Xc_{\emptyset}(t)>\epsilon$, we observe from the remarks above that there exists a certain $w \in \Rm$ such that $t  = \beta_{w}+ \dtb_{w}$ and $X_{w\vr}(0) = y$. So a daughter cell is born at $t$ and its size evolves according to $\A^X_{w\vr \Rm}$, which is the process associated with $w\vr\Rm:= \{ w\vr v, v\in \Rm\}$ as in \eqref{eq:AX}. As $w\vr\Rm$ is the left-most branch in the sub-tree $(w\vr v, v\in \Utwo)$, we deduce from Lemma~\ref{lem:Heta} that $\A^X_{w\vr \Rm}$ has distribution $P_{y}$ and $\A_{w\vr \Rm}$ defined as in \eqref{eq:A} is $\A^X_{w\vr \Rm}$ killed when entering $(0,\epsilon]$.
On the other hand, for every time $t'\leq \zeta_{\Rm}$ with $y':=-\Delta \Xc_{\emptyset}(t')\in(0, \epsilon]$, we agree that the daughter cell born at $t'$ is killed immediately. We hence conclude that those non-degenerate size processes $\cutc{\Xc}_i$ with $i\in \N$ are exactly those non-degenerate processes $\A_{w\vr \Rm}$ with $w\in \Rm$. The proof is completed by iteration of this argument.
\end{proof}

\begin{proof}[Proof of Theorem \ref{thm:1}]
For every $\epsilon>0$, applying Lemma \ref{lem:WX} to $\Ws_{(Y,X)}$, we deduce that $\Ws_{(Y,X)}$ and $\cutc{\Ys}$ have the same law. Together with Lemma \ref{lem:WW}, this implies that $\cutc{\Xs}$ and $\cutc{\Ys}$ have the same law. Letting $\epsilon \to 0+$, we conclude by Lemma \ref{lem:CV} that $\Xs$ and $\Ys$ have the same finite-dimensional distribution. 
 \end{proof}


\subsection{Proof of Theorem \ref{thm:2'}}\label{sec:proofs}
Using Theorem \ref{thm:1}, we complete the proof of Theorem \ref{thm:2'}. 

\begin{proof}[Proof of Theorem \ref{thm:2'}] 
The implication $(i)\Rightarrow (ii)$ follows from Lemma \ref{lem:bifH} and the equivalence $(iii)\Leftrightarrow (i)$ follows from Corollary \ref{cor:prop1}. So it remains to prove that $(ii)\Rightarrow (iii)$. Suppose that $X^{(0)}$ and ${\tilde{X}}^{(0)}$ can be coupled to form a bifurcator.  We can check that $X^{(0)}$ and ${\tilde{X}}^{(0)}$ satisfy \ref{H} and \ref{Heta}, then we are led to the conclusion that the growth-fragmentations have the same finite dimensional distribution by Theorem \ref{thm:1}. Indeed, fix $q\geq 2$ and $K > \kappa(q)$, then we know from Example \ref{eg:hom} that $X^{(0)}$ satisfies \ref{H} with the function $(t,x)\mapsto x^q \e^{-K t}$; further, it follows from \eqref{eq:hom} that $X^{(0)}$ also satisfies \ref{Heta} with this function and any $\eta \in ( \frac{\kappa(q) -\Phi(q)}{K - \Phi(q)}, 1)$. Similarly, we have that ${\tilde{X}}^{(0)}$ also satisfies both \ref{H} and \ref{Heta}. This completes the proof. 
\end{proof}

\subsection{Proof of Theorem \ref{thm:2}}\label{sec:ss}

We now turn to self-similar growth-fragmentations. In order to prove Theorem \ref{thm:2} (and thus close this paper), we first prove the following lemmas.  

\begin{lem}\label{lem:ssp}
Let $X^{(\alpha)}$  be a self-similar cell process with index $\alpha\in \R$ related to a SNLP $\xi$ as in \eqref{eq:Lamperti}. Suppose $\kappa(q)<0$ for a certain $q>0$, then $X^{(\alpha)}$ satisfies both \ref{H} and \ref{Heta} (for any $\eta\in (  1 - \frac{\kappa(q)}{\Phi(q)}, 1)$) with the function $(t,x)\mapsto x^q$.  
\end{lem}
\begin{proof}
This follows directly from Lemma 2 and Lemma 3 in \cite{Bertoin:growth}.
\end{proof}

\begin{lem}\label{lem:bif}
Let $X^{(\alpha)}$ and $Y^{(\alpha)}$ be two self-similar cell processes with index $\alpha\in \R$ related to SNLPs $\xi$ and $\gamma$ respectively as in \eqref{eq:Lamperti}. Suppose that $\kappa = \kappa_{\gamma}$, then $X^{(\alpha)}$ and $Y^{(\alpha)}$ can be coupled to form a bifurcator. 
\end{lem}

\begin{proof}  
By Proposition \ref{prop:st} we may assume that $\xi$ is the switching transform of $\gamma$ with switching time $\tb= \inf\{t\geq 0:~ \xi(t)\neq \gamma(t)\}$. Say $X^{(\alpha)}(0)=Y^{(\alpha)}(0)=x>0$, we set  
$$\tb^{(\alpha)}:= x^{-\alpha}\int_0^{\tb} \exp(-\alpha \xi(r))dr ,$$
then we have by Lamperti's time-substitution \eqref{eq:Lamperti} that $\tb^{(\alpha)} = \inf \{t\geq 0: X^{(\alpha)}(t) \neq \tildeX^{(\alpha)}(t)\}$ and $X^{(\alpha)}(\tb^{(\alpha)})+ \tildeX^{(\alpha)}(\tb^{(\alpha)}) =X^{(\alpha)}(\tb^{(\alpha)}-)$. 
Let $\tilde{Y}^{(\alpha)}$ be an independent copy of $Y^{(\alpha)}$ and we build a process
$$\hat{Y}^{(\alpha)} (t):=  Y^{(\alpha)}(t)\ind{t< \tb^{(\alpha)}} + y \tilde{Y}^{(\alpha)}(y^{\alpha}(t-\tb^{(\alpha)}))\ind{t\geq \tb^{(\alpha)}} ,\quad t\geq 0,$$
where $y:= -\Delta X^{(\alpha)}(\tb^{(\alpha)})$. Then $(X^{(\alpha)},\hat{Y}^{(\alpha)})$ is a bifurcator.  
\end{proof}

We now complete the proof of Theorem \ref{thm:2}. 
\begin{proof}[Proof of Theorem \ref{thm:2}]
$(i) \Rightarrow (ii)$: This follows from Lemma \ref{lem:bif}. 

$(ii)\Rightarrow (iii)$: Since Lemma \ref{lem:ssp} ensures that \ref{H} and \ref{Heta} hold under assumption \eqref{eq:sspH}, we have from Theorem \ref{thm:1} that the self-similar growth-fragmentations $\Xs^{(\alpha)}$ and $\tilde{\Xs}^{(\tilde{\alpha})}$ have the same finite-dimensional distribution. 

$(iii)\Rightarrow (i)$: Suppose that the growth-fragmentations $\Xs^{(\alpha)}$ and $\tilde{\Xs}^{(\tilde{\alpha})}$ have the same finite-dimensional distribution. We first know from the self-similarity (Theorem 2 in \cite{Bertoin:growth}) that $\alpha= \tilde{\alpha}$. 

To prove $\kappa = \tilde{\kappa}$, we first deduce from Proposition 5 and its proof in \cite{BCK:martingales} that for every $q>0$, there is 
\begin{equation}\label{eq:pot}
 \Exps[1]{\int_0^{\infty} \left(\sum_{y\in \Xs^{(\alpha)}(t)} y^{q+\alpha} \right) dt} =\begin{cases}
    -\frac{1}{\kappa(q)}, & \text{if } \kappa(q)<0, \\
    \infty, & \text{otherwise}.
  \end{cases}
\end{equation}
Note from Corollary 4 in \cite{Bertoin:growth} that the integrand possesses c\`adl\`ag paths under assumption \eqref{eq:sspH}. As $\Xs^{(\alpha)}$ and $\tilde{\Xs}^{(\alpha)}$ have the same same finite-dimensional distribution, we thus deduce that for every $q>0$ with $\kappa(q)<0$, there is $\tilde{\kappa}(q)= \kappa(q)<0$. 
Therefore, if there exists $q_0>0$ such that $\kappa(q)<0$ for all $q>q_0$, then $\kappa(q) = \tilde{\kappa}(q)$ for all $q>q_0$. Otherwise, by the convexity of $\kappa$ these exists $\omega>0$, which is the largest root of $\kappa$, such that $\kappa(q)>0$ for all $q>\omega$. It follows from \eqref{eq:pot} that $\omega$ is also the largest root for $\tilde{\kappa}$ and $\tilde{\kappa}(q)>0$ for all $q>\omega$. 
We hence deduce from Theorem 7 in \cite{BCK:martingales} that $\tilde{\kappa}(q) = \kappa(q)$ for all $q>\omega$. 
Summarizing the two cases, we conclude that there exists a certain constant $a>0$ such that $\kappa(q) = \tilde{\kappa}(q)$ for all $q >a$, which entails that $\kappa = \tilde{\kappa}$.  
\end{proof}


\bibliographystyle{amsplain}
\bibliography{quanshi_embedding.bib}

\end{document}